\documentclass[11pt]{article}

\usepackage{srcltx}
\usepackage{epsfig}
\usepackage{url}
\usepackage{bbm}
\usepackage{float}
\usepackage{cancel}
\usepackage{cellspace}
\usepackage{makecell}
\usepackage{slashbox}
\usepackage{fullpage}

\def\diag{{\hbox{diag}}}

\def\bZ{{\mathbf{Z}}}

\def\cC{{\cal C}}
\def\cL{{\cal L}}

\usepackage{psfrag}
\usepackage{graphicx}
\usepackage{hyperref}
\usepackage{mathtools}
\usepackage{amsfonts,amsmath,amssymb,amsthm}
\usepackage{color}
\usepackage{booktabs}

\usepackage{float}

\newcommand{\R}{\mathbb{R}}

\floatstyle{ruled}
\newfloat{algorithm}{tbp}{loa}
\floatname{algorithm}{Algorithm}

\DeclareMathOperator*{\argmin}{Arg\,min}
\newtheorem{thm}{Theorem}%[section]

\newtheorem{lemma}[thm]{Lemma}

\def\qed{\ \hfill$\square$\par\smallskip}

\def\Opt{{\mathop{\hbox{Opt}}}}

\newcommand{\cN}{I\!\! N}

\newcommand{\C}{{\cal C}}

\newcommand{\N}{{\cal N}}

\newcommand{\half}{ \mbox{\small$\frac{1}{2}$}}

\def\e{\epsilon}

\def\Prob{{\hbox{\rm Prob}}}

\newcommand{\be}{\begin{eqnarray}}
\newcommand{\ee}[1]{\label{#1}\end{eqnarray}}

\newcommand{\ese}{\end{eqnarray*}}
\newcommand{\bse}{\begin{eqnarray*}}
\newcommand{\rf}[1]{~(\ref{#1})}

\newtheorem{proposition}{Proposition}%[section]
\newtheorem{corollary}{Corollary}%[section]
\newtheorem{remark}{Remark}%[section]
\newtheorem{definition}{Definition}%[section]
\def\argmin{\mathop{\hbox{\rm argmin$\,$}}}
\def\Argmin{\mathop{\hbox{\rm Argmin$\,$}}}

\def\Card{{\hbox{\rm Card}}}

\def\cN{{\cal N}}

\def\bR{{\mathbf{R}}}

\def\bE{{\mathbf{E}}}
\def\e{{\rm e}}
\def\Opt{{\hbox{\rm Opt}}}
\newcommand{\transp}{{\scriptscriptstyle \top}}
\newcommand{\corri}[2]{{\color{blue}#2}}

\newcommand{\smax}{\sigma_{\max}}

\begin{document}
\title{Non-asymptotic confidence bounds for the optimal value of a stochastic program}

 %\author{
% \name{Vincent Guigues\textsuperscript{a}$^{\ast}$\thanks{$^\ast$Corresponding author. Email: vguigues@fgv.br}
% and Anatoli Juditsky\textsuperscript{b} and Arkadi Nemirovski\textsuperscript{c}}
% \affil{\textsuperscript{a}FGV/EMAp, 22 250-900 Rio de Janeiro, Brazil;
% \textsuperscript{b}LJK, Universit\'e Grenoble Alpes, B.P. 53, 38 041 Grenoble Cedex 9, France, {\tt anatoli.juditsky@imag.fr}
% \textsuperscript{c}Georgia Institute of Technology, Atlanta, Georgia 30 332, USA, {\tt nemirovs@isye.gatech.edu}}}
%%\received{v5.0 released July 2015}
\author{
Vincent Guigues
\thanks{FGV/EMAp, 190 Praia de Botafogo, Botafogo, 22250-900 Rio de Janeiro, Brazil, {\tt vguigues@fgv.br}}
\and
Anatoli Juditsky
\thanks{LJK, Universit\'e Grenoble Alpes, B.P. 53, 38041 Grenoble Cedex 9, France,	
{\tt anatoli.juditsky@imag.fr}}
\and
Arkadi Nemirovski
\thanks{Georgia Institute
 of Technology, Atlanta, Georgia
30332, USA, {\tt nemirovs@isye.gatech.edu}\newline
Research of the first author was supported by an FGV grant, CNPq grant 307287/2013-0,
FAPERJ grants E-26/110.313/2014, and E-26/201.599/2014.
The second author was supported by the CNRS-Mastodons project GARGANTUA,
and the LabEx PERSYVAL-Lab (ANR-11-LABX-0025). Research of
the third author was supported by NSF grants {CMMI-1232623, CMMI-1262063, CCF-1415498.}}}

\maketitle

\begin{abstract}
We discuss a general approach to {building} non-asymptotic confidence bounds for stochastic optimization problems. Our principal contribution is the observation that a
{Sample} Average Approximation of a problem supplies upper and lower bounds for the optimal value of the problem which are essentially better than the quality of
the corresponding optimal solutions. At the same time, such bounds are more reliable than ``standard'' confidence bounds obtained through the asymptotic approach.
We also discuss bounding the optimal value of MinMax Stochastic Optimization and stochastically constrained problems. We conclude with a
simulation study illustrating the numerical behavior of the proposed bounds.
\end{abstract}
%

%\begin{keywords}
%Sample Average Approximation; confidence interval; MinMax Stochastic Optimization; stochastically constrained problems.
%\end{keywords}
%
%\begin{classcode}90C15, 90C90, 90C30\end{classcode}

\section{Introduction}
Consider the following Stochastic Programming (SP) problem
\be
\Opt = \min_{x}[f(x)= \bE\{F(x, \xi)\},\;x\in X]
\ee{eq:opt1}
where $X$ is a nonempty bounded closed convex set of a Euclidean space $E$, $\xi$ is a random vector with probability distribution
$P$ on  $\Xi\subset \bR^k$ and $F:\; X\times \Xi\to \bR$. There are two competing approaches for solving \rf{eq:opt1} when a sample $\xi^N=(\xi_1,...,\xi_N)$ of realizations of $\xi$ (or a device to sample from the distribution $P$) is available --- Sample Average Approximation
(SAA) and the Stochastic Approximation (SA).
The basic idea of the SAA method is to  build
an approximation of the ``true'' problem \rf{eq:opt1} by replacing the expectation $f(x)$ with its sample average approximation
\[
f_N(x, \xi^{N})={1\over N} \sum_{t=1}^NF(x,\xi_t),\; x\in X.
\]
The resulting optimization problem {has} been extensively studied theoretically  and numerically (see, e.g., \cite{kleywegt2002sample,linderoth2006empirical,mak1999monte,shapiro2003monte,shapiro2005complexity,verweij2003sample},
among
many others). In particular, it was shown that the SAA method (coupled with a deterministic algorithm for minimizing the
SAA) is often efficient for solving  large classes of stochastic programs. The alternative SA approach was also extensively studied since the pioneering work
by Robbins and Monro \cite{robbins1951stochastic}. Though possessing better theoretical accuracy estimates, SA was long time believed to underperform numerically. It
was recently demonstrated (cf., \cite{StAppr,bottou2010large,srebro2010stochastic}) that a proper modification of the SA approach, based
on the ideas behind the  Mirror Descent algorithm \cite{nemirovsky1983problem} can be competitive and can even significantly
outperform the SAA method on a large class of convex stochastic programs.

Note that  in order to qualify the accuracy of approximate solutions (e.g., to build  efficient stopping criteria) delivered by the stochastic algorithm of choice, one needs to
construct lower and upper bounds for the optimal value $\Opt$ of problem \rf{eq:opt1} from stochastically sampled observations. Furthermore, the question of computing reliable upper and, especially, lower bounds for the optimal value is of interest in many applications.
Such bounds  allow  statistical decisions (e.g., computing confidence intervals, testing statistical hypotheses) about the optimal value. For instance, using the approach to regret minimization, developed in \cite{perchet2013-2,perchet2013}, they may be used to construct {\em risk averse strategies} for multi-armed bandits, and so on.

An important methodological feature of the SAA approach is its asymptotic framework which explains how
to provide asymptotic estimates of the accuracy of the obtained solution by computing asymptotic upper and lower
bounds for the optimal value of the ``true'' problem (see, e.g., \cite{dupawets1988,shap1991,kingrock1993,pflug1995asy,mak1999monte,pflug1999stochastic,pflug2003}, and references therein).
% \corri{Note that
% the non-asymptotic accuracy of optimal solutions of the SAA problem has been analyzed in \cite{kanikingwets}  and
% more recently in \cite{pflug2003,shapiro2003monte,shapiro2005complexity}, yet, to the best of our knowledge, the literature
% does not provide any {\em online} (computed only in terms of the sample used to solve SAA without using an additional, offline, sample) {\em non-asymptotic} construction of lower and upper bounds for the optimal value of \rf{eq:opt1} by SAA.
% On the other hand, non-asymptotic lower and upper bounds for the objective value were built in \cite{lan2012validation} using Stochastic Mirror Descent (SMD).
% and in \cite{guigues2016mp} using multistep SMD (both for risk-neutral and risk-averse problems).}

However,
as is always the case with techniques which are validated asymptotically, some important questions, such as ``true'' reliability of  bounds, cannot be answered by the asymptotic analysis. 
Note that the non-asymptotic accuracy of optimal solutions of the SAA problem was recently analysed (see, e.g., \cite{kanikingwets,pflug1999stochastic,pflug2003,shapiro2003monte,shapiro2005complexity,shalev2009learnability}), 
yet, to the best of our knowledge, the literature does not provide any  {\em non-asymptotic} construction of lower and upper bounds for the optimal value of \rf{eq:opt1} by SAA.
On the other hand, non-asymptotic lower and upper bounds for the objective value by SA method were built in \cite{lan2012validation} and \cite{guigues2016mp}.

Our objective in this work is to fill this gap, by building reliable finite-time evaluations of the optimal value of \rf{eq:opt1}, which are also good enough to be of practical interest. Our basic methodological observation is Proposition \ref{propmain} which states that the SAA of problem \rf{eq:opt1} comes with a ``built-in'' non-asymptotic lower and upper estimation of the ``true'' objective value. The accuracy of these estimations is essentially higher than
the available theoretical estimation of the {\em quality of the optimal solution} of the SAA. Indeed, when solving
a high-dimensional SAA problem, the (theoretical bound of) inaccuracy of the optimal solution becomes a function of dimension. In particular, when the set $X$ is a unit Euclidean ball of $\bR^n$, the accuracy of the SAA {\em optimal solution}
{may} be by factor $O(n)$ worse than the corresponding accuracy of the SA solution \cite{StAppr}. In contrast to this, the {\em optimal value} of the SAA problem supplies an approximation of the ``true'' optimal value of accuracy which is (almost) independent of
problem's dimension and may be used to construct non-trivial non-asymptotic confidence bounds for the true optimal value. This fact is surprising, because the bad
theoretical accuracy bound for optimal solutions of SAA reflects their actual
behavior on some problem instances (see
{Proposition \ref{pro:lowersaa} and the discussion in Section \ref{sec:discussion}}).
%\footnote{It is worth to mention that it coincides, up to an absolute factor, with the accuracy of the approximate solution to the stochastic problem supplied by the Stochastic Approximation (see, e.g.,  \cite{StAppr}). }

% \corri{We also consider convex expectation constrained stochastic problems of the form
% \begin{equation}\label{CSPintro}
% \Opt=\min_{x\in X}\left[ \bE\{  F_0(x,\xi) \} :\;\; \bE\{ F_i(x,\xi) \} \leq 0,\,1\leq i\leq m\right],
% \end{equation}
% where $X$ is a nonempty bounded closed convex set of a Euclidean space $E$, $\xi$ is a random vector with probability distribution
% $P$ on  $\Xi\subset \bR^k$, and
% $
% F_i(x,\xi):E\times\Xi\to \bR,\, 0\leq i\leq m.
% $
% For such problems, we show how confidence bounds can be constructed for an {\em $\epsilon$-underestimation} of the optimal value.
% }{}
The paper is organized as follows. 

% \corri{We present the construction of lower and upper confidence bounds for the optimal value of stochastic problem \eqref{eq:opt1}
% in Section \ref{sec:theoryuc}.
% Then in Section \ref{ccase} we derive bounds for constrained problems of form \eqref{CSPintro} and
% build lower and upper bounds for the optimal value of MinMax Stochastic Optimization.
% \par
% Finally, several
% simulation experiments illustrating the properties of the bounds built in Sections \ref{sec:theoryuc} and \ref{ccase} are presented in Section \ref{sec:num}.  Proofs of theoretical statements are
% collected in the appendix.
% }

We present the construction of lower and upper confidence bounds for the optimal value of a stochastic problem in Section \ref{sec:theory}. Specifically,
in Section \ref{sec:principal}, we
develop confidence bounds for the optimal value of problem \rf{eq:opt1}.
Then in Section \ref{ccase} we build lower and upper bounds for the optimal value of MinMax Stochastic Optimization
and show how the confidence bounds can be constructed for an {\em $\epsilon$-underestimation} of the optimal value of a (stochastically) constrained Stochastic Optimization problem.

Finally, several
simulation experiments illustrating the properties of the bounds built in Section \ref{sec:theory} are presented in Section \ref{sec:num}.  Proofs of theoretical statements are
collected in the appendix.
%

%{\color{red}

\section{Confidence bounds via Sample Average Approximation}\label{sec:theory}

\subsection{Problem without stochastic constraints}\label{sec:principal}

\subsubsection{Situation.}\label{sec:situation} In the sequel, we fix a Euclidean space $E$ and a norm $\|\cdot\|$ on $E$. We denote
{by}
$B_{\|\cdot\|}$ the unit ball of the norm $\|\cdot\|$, and by $\|\cdot\|_*$ the norm conjugate to $\|\cdot\|$:
$$
\|y\|_*=\max_{\|x\|\leq {1}}\langle x,y\rangle.
$$
Let us now assume that we are given a function $\omega(\cdot)$ which is continuously differentiable on $B_{\|\cdot\|}$ and strongly convex with respect to 
$\|\cdot\|$, with  parameter of strong convexity equal to one, i.e., such that
and for every  $x,y \in B_{\|\cdot\|}$
\[
(\nabla \omega(x) - \nabla \omega(y))^T (x-y) \geq \|x-y\|^2,
\]
with
$\omega(0)=0$ and $\omega'(0)=0$ (in other words, $\omega(\cdot)$ is a {\em distance-generating function compatible with $\|\cdot\|$}).  We denote
\begin{equation}\label{Omega}
\Omega=\max_{x:\|x\|\leq1}\sqrt{2 \omega(x) }.
\end{equation}
Let, further,\begin{itemize}
 \item $X$ be a convex compact subset of $E$,
 \item $R=R_{\|\cdot\|}[X]$ be the smallest radius of a $\|\cdot\|$-ball containing $X$, 
 \item $P$ be a Borel probability distribution on  $\bR^{k}$, $\Xi$ be the support of $P$, and
 $$
 F(x,y): \;E\times\Xi\to\bR
 $$
 be a Borel function which is convex in $x\in E$ and is $P$-summable for every $x\in E$, so that the function
$$
f(x)=\bE\{F(x,\xi)\}:E\to\bR
$$
is well defined and convex.
\end{itemize}
We denote
\[
L(x,\xi)=\max\left\{\|g-h\|_{*}: g\in \partial_xF(x,\xi),h\in\partial f(x)\right\}.
\]
The outlined data give rise to the stochastic program
\[%\begin{equation}\label{SP}
\Opt=\min_{x\in X}\left[f(x)=\bE\{F(x,\xi)\}\right]
\]
and its {\sl Sample Average Approximation} (SAA)
\begin{equation}\label{SAAdef}
\Opt_N(\xi^N)=\min_{x\in X} \left[f_N(x,\xi^N):={1\over N} \sum_{t=1}^NF(x,\xi_t)\right],
\end{equation}
where $\xi^N=(\xi_1,...,\xi_N)$, and $\xi_1,\xi_2,...$ are drawn  independently from $P$. Our immediate goal is to understand how well the optimal value $\Opt_N(\xi^N)$ of SAA approximates the true optimal value $\Opt$.

\subsubsection{Confidence bounds}\label{sec:subprincipal}
Our main result is as follows.
\begin{proposition}\label{propmain} 
In the situation of Section \ref{sec:situation},  let us assume that $f$ is differentiable on $X$ and that for some positive $M_1$, $M_2$ and all $x\in X$ one has
\begin{equation}\label{bounds}
\begin{array}{llcll}
(a)& \bE\Big[\e^{(F(x,\xi)-f(x))^2/M_1^2} \Big] \leq \e,&\quad&
(b)& \bE\Big[\e^{L^2(x,\xi)/M_2^2} \Big] \leq \e.
\end{array}
\end{equation}
Define \[%\begin{equation} \label{defab}
a(\mu,N) = \frac{\mu M_1}{\sqrt{N}}\;\;\; \mbox{ and }\;\;\;b(\mu, s, \lambda,N)=\frac{\mu M_1 +\left[\Omega[1+s^2]+2\lambda\right]M_2 R}{\sqrt{N}},
\]%\end{equation}
where $\Omega$ is as in \rf{Omega}, and let
$
\tau_*=0.557409...$ be the smallest positive real such that \mbox{$\e^t\leq t+\e^{\tau_* t^2}$} for all $t\in\bR$.
Then for all $N\in \bZ_+$ and $\mu\in[0,2\sqrt{\tau_*N}]$
\begin{equation}\label{upper}
\Prob\Big\{\Opt_N(\xi^N)>\Opt+  a(\mu,N) \Big\}\leq \e^{-{\mu^2\over 4\tau_*}};
\end{equation}
and for all $N\in\bZ_+,\;\mu\in[0,2\sqrt{\tau_*N}],\;s>1$ and $\lambda\ge 0$,
\begin{equation}\label{lower}
\Prob\Big\{\Opt_N(\xi^N)<\Opt- b(\mu, s, \lambda,N)  \Big\}\leq
\e^{-N(s^2-1)}+\e^{-{\mu^2\over 4\tau_*}}+\e^{-{\lambda^2\over 4\tau_*}}.
\end{equation}
\end{proposition}
We have the following obvious corollary to this result.
\begin{corollary}\label{vinc}
Under the premise of Proposition  {\ref{propmain}}, let
%\begin{equation} \label{defab}
\[
\begin{array}{rcl}
{\tt{Low}}^{\tt{SAA}}(\mu_1,N) &=&  \Opt_N(\xi^N)-a(\mu_1,N),\\
{\tt{Up}}^{\tt{SAA}}(\mu_2, s, \lambda,N)&=&\Opt_N(\xi^N) +b(\mu_2,s,\lambda,N).
\end{array}
\]
%\end{equation}
Then for all $N \in \bZ_+,\; s>1,\,\lambda\geq 0,\,\mu_1, \mu_2 \in[0,2\sqrt{\tau_*N}]$
\[%begin{equation}\label{confinterval}
\Prob\Big\{\Opt \in \Big[{\tt{Low}}^{\tt{SAA}}(\mu_1,N), \,{\tt{Up}}^{\tt{SAA}}(\mu_2, s, \lambda,N) \Big] \Big\}\geq 1 -  \beta%(\mu_1,\mu_2,s,\lambda,N)
\]%end{equation}
where $\beta[=\beta(\mu_1,\mu_2,s,\lambda,N)]
=\e^{-{\mu^2_1\over 4\tau_*}}+\e^{-{\mu^2_2\over 4\tau_*}}+ \e^{-N(s^2-1)}+  \e^{-{\lambda^2\over 4\tau_*}}$.
In other words, for the choice of $\mu_1,\mu_2,s,\lambda$ and $N$ such that $0<\beta<1$, the segment $[{\tt{Low}}^{\tt{SAA}}, \,{\tt{Up}}^{\tt{SAA}}]$ is the confidence interval for $\Opt$ of level $1-\beta$.
\end{corollary}
%\paragraph{Proof.} The proof is given in the Appendix.\qed

\subsubsection{Discussion.}\label{sec:discussion}
The result of Proposition \ref{propmain}  merits some comments.
\begin{enumerate}
\item Confidence bounds of Proposition \ref{propmain} and Corollary \ref{vinc} involve constants $M_1$ and $M_2$, defined in \rf{bounds}. {Valid upper bounds on} these constants are crucial to obtain sound confidence bounds. To the best of our knowledge there is no generic procedure which 
allows us to construct such estimates. Nevertheless, it is possible to build ``reasonably good'' bounds for $M_1$ and $M_2$ in specific problem settings. For instance, we provide such bounds for the examples used to illustrate the results of this section in the numerical experiments of Section \ref{sec:num} (see Appendix \ref{B-appendix} for details of the calculations).

\item``As is'', Proposition \ref{propmain} requires $f(\cdot)$ to be differentiable. This purely technical assumption is in fact not restrictive at all.
Indeed, we can associate with (\ref{eq:opt1}) its ``smoothed'' approximation
$$
\min_{x\in X} \left[f_\epsilon(x):=\int_{\Xi\times {E  }}F_\epsilon(x,[\upsilon;\xi])P(d\xi)p(\upsilon)d \upsilon\right],\,\,F_\epsilon(x,[\upsilon;\xi])=F(x+\epsilon \upsilon,\xi),
$$
where $p(\cdot)$ is, say, the density of the uniform distribution $U$ on the unit ball $B_\infty$  in $E$. Assuming 
that bounds (\ref{bounds}.a) and (\ref{bounds}.b) hold  for all $x$ from an open set $X^+$ containing $X$, it is immediately seen that $f_\epsilon$ is, for 
values of $\epsilon$ small enough, a  continuously differentiable function  on $X$ which converges, uniformly on $X$, to $f$ as $\epsilon\to {+0}$. Given a possibility to sample from the distribution $P$, we can
sample from the distribution $P^+:=P\times U$ on $\Xi^+=\Xi\times {E}$, and thus can build the SAA of the problem
$\min_x f_\epsilon(x)$.
When $\epsilon$ is small, this smoothed problem satisfies the premise of Proposition \ref{propmain}, the parameters $M_1$, $M_2$ remaining unchanged, and its optimal value can be made as close to $\Opt$ as we wish by an appropriate choice of $\epsilon$. As a result, by passing from the SAA of the original problem to the SAA of the smoothed one, $\epsilon$ being small, we ensure, ``at no cost,'' smoothness of the objective, and thus -- applicability of the large deviation bounds stated in Proposition \ref{propmain}.

\item The standard theoretical results on the SAA of a stochastic optimization problem (\ref{eq:opt1}), see, e.g. \cite{StAppr,shapiro2014} and references therein, are aimed at quantifying the sample size 
$N=N(\epsilon ,n)$ which, with overwhelming probability, ensures that an {\em optimal solution} $x(\xi^N)$ to the SAA of the problem of interest satisfies the relation $f(x(\xi^N))\leq \Opt+\epsilon$, for a given $\epsilon>0$. The corresponding bounds on $N$ are similar, but not identical, to the bounds in Proposition~\ref{propmain}. Let us consider, for instance, the simplest case of ``Euclidean geometry'' where \mbox{$\|x\|=\|x\|_2=\sqrt{\langle x,x\rangle}$}, $\omega(x)={1\over 2}\|x\|^2$, and
$X$ is the unit $\|\cdot\|_2$-ball. In this case Proposition~\ref{propmain} states that for a given $\epsilon>0$, the sample size $N$ for which $\Opt(\xi^N)$ is, with probability at least $1-\alpha$, $\epsilon$-close to $\Opt$,  can be upper-bounded for small enough $\epsilon$ and $\alpha$ by
\[
N_\epsilon:=C  {[M_1+M_2]^2\ln(1/\alpha)\over \epsilon^2}
\]
(here $C$ is a positive absolute constant).\footnote{E.g., for $\alpha\leq \half$ and $\epsilon \leq M_1+4M_2$ we have an immediate (though rough) bound
\[
N_\epsilon={4\tau_*(M_1+4M_2)^2\ln(4/\alpha)\over \epsilon^2}.
\]
}
It should be stressed that both the bound itself and the range of ``small enough'' values of $\epsilon,\,\alpha$ for
which this bound is valid are independent of the dimension $n$ of the decision vector $x$. In contrast to this, available estimation of the complexity 
$N(\epsilon, n)$ relies upon uniform convergence arguments and is affected by problem's dimension: up to
logarithmic terms, $N(\epsilon, n)=n N_\epsilon$ (cf. the discussion in \cite{shapiro2005complexity,shalev2009learnability}). This phenomenon -- linear dependence on the problem's dimension $n$ of the SAA sample size yielding, with high probability, an $\epsilon$-optimal solution to a stochastic problem -- is
not an artifact stemming from {an} imperfect theoretical analysis of the SAA but reflects the actual performance of SAA on some instances. Indeed, we have the following:
\begin{proposition}\label{pro:lowersaa}
For any $n\geq 3$, and $R,L>0$  one can point out a convex Lipschitz continuous function $f$ with Lipschitz constant $L$  on the Euclidean ball $B_2(R)$ of radius $R$, and
 an integrand $F(x,\xi)$ convex in $x$ such that $\bE_\xi\{F(x,\xi)\}=f(x)$, $\|F'(x,\xi)-f'(x)\|_2^2\leq L$ a.s., for all $x\in B_2(R)$, and such that with probability at least $1-\e^{-1}$ there is an optimal
solution ${x(\xi^N)}$ to the SAA
\[
\min\left[f_N(x, \xi^N )={1\over N} \sum_{i=1}^NF(x,\xi_i):\;x\in B_2(R)\right],
\]
sampled over $N\leq n$ i.i.d. realizations of $\xi$, satisfying
\begin{equation}\label{bdSAA}
f({x(\xi^N)})-{\Opt} \geq c_0 LR,
\end{equation}
where $c_0$ is a positive absolute constant.

\end{proposition}
Note that for large-scale problems, the presence of the factor $n$ in the sample size bound is a definite and serious drawback of SAA. A nice fact about the SAA approach
 as expressed by Proposition \ref{propmain}, is that {\sl as far as reliable $\epsilon$-approximation of the optimal value
{\em (rather than building an $\epsilon$-solution)} is concerned, the performance of the SAA approach}, at least in the {case of favorable geometry,} {\sl is not affected by the problem's dimension.} It should be stressed that the crucial role in  Proposition \ref{propmain} is played by convexity which allows us to express the quality to which the SAA
reproduces the optimal value in (\ref{eq:opt1}) in terms of how well $f_N(x,\xi^N)$ reproduces the first order information on $f$ {\sl at a single point} $x_*\in\Argmin_X f$, see the proof of Proposition \ref{propmain}.
In a ``favorable geometry'' situation, e.g., in the Euclidean geometry case, the corresponding sample size is not affected by problem's dimension. In contrast to this, to yield reliably an $\epsilon$-solution, the
SAA requires, in general, $f_N(x,\xi^N)$ to be $\epsilon$-close to $f$ {\sl uniformly on $X$} with overwhelming probability; and the corresponding sample size, even in the case of Euclidean geometry, grows
with problem's dimension.
\item
Note that (at least in the case of Euclidean geometry) {without additional}, as compared to those in Proposition \ref{propmain}, {restrictions on $F$ and/or the distribution $P$, the quality of the SAA estimate
$\Opt_{N}(\xi^N)$ of $\Opt$ (and thus, the quality of the confidence interval for it provided by Corollary \ref{vinc}) is}, within an absolute constant factor, {\sl the best allowed by the laws of Statistics.}
Namely, we have the following lower bound for the widths of {the} confidence intervals {for the} optimal value valid already for a class of {\em linear} stochastic problems.
\begin{proposition}\label{pro:lower2arik}
For any $n\geq 1$, $M_1\geq M_2>0$, one can point out a family of linear stochastic optimization problems, i.e., linear functions $f$ on the unit Euclidean ball $B_2$ of $\bR^n$ and 
corresponding integrands $F(x,\xi)$  linear in $x$ such that $\bE_\xi\{F(x,\xi)\}=f(x)$, satisfying the premises of Proposition \ref{propmain} and Corollary \ref{vinc}, and such that
the width of the confidence interval for $\Opt=\min_{x\in B_2} f(x)$ of confidence level $\geq 1-\alpha$ cannot be less than
\be
{\underbar{W}}=2\gamma q_\N(1-\alpha) { M_1\over \sqrt{N}},
\ee{underbar}
where $q_\N(\beta)$ is the $\beta$-quantile of the standard normal distribution,
and $\gamma>0$ is given by the relation
$$
\bE_{\zeta\sim \cN(0,1)}\left\{\exp\{\gamma^2\zeta^2\}\right\}=\exp\{1\},
$$
or, equivalently,
$
\gamma^2=\half(1-\exp\{-2\}).
$
\end{proposition}

In Table \ref{arik_exp}, we provide the ratios $R_{\rm W}$
of the widths of the confidence intervals, as given by Corollary \ref{vinc} and their lower bounds
for some combinations of risks  $\alpha$ and parameters $M_1, M_2$ and $N$.
%$R_{\rm W}=\underbar{W}^{-1}[{\tt{Up}}^{\tt{SAA}}(\mu_2, s, \lambda,N) - {\tt{Low}}^{\tt{SAA}}(\mu_1,N)] $
%where ${\tt{Up}}^{\tt{SAA}}(\mu_2, s, \lambda,N)$ and ${\tt{Low}}^{\tt{SAA}}(\mu_1,N)$
%given in Corollary \ref{vinc}
%and ${\underbar{W}}$ is as in \rf{underbar}.
%These ratios are computed taking
%$\Omega=R=1$, $\mu_1 = 2 \sqrt{\alpha_* \ln(\frac{2}{\alpha})}$, $\mu_2 = \lambda = 2 \sqrt{\alpha_* \ln(\frac{2}{\alpha})}$,
%$s=\sqrt{1+\frac{1}{N}\ln(6/\alpha)}$,
%and different combinations of risks  $\alpha$ and parameters $M_1$, $M_2$, and $N$.

\begin{table}
\centering
\begin{tabular}{|c|c|c|c||c|c|c||c|c|c|}
%\cline{2-10}
%\multicolumn{1}{c}{}
\hline
$\alpha=0.1$&\multicolumn{3}{c||}{$M_1=M_2=1$}&\multicolumn{3}{c||}{$M_1=10, \,M_2=1$}&\multicolumn{3}{c|}{$M_1=100,\,M_2=1$}\\
\hline
$N$&10&100&1000&10&100&1000&10&100&1000\\\hline
$R_{\rm W}$&8.086 & 7.803 & 7.775 &3.772 &3.744 &3.741 &3.341& 3.338& 3.337\\
\hline
\end{tabular}
\\~\\
\begin{tabular}{|c|c|c|c||c|c|c||c|c|c|}
%\cline{2-10}
%\multicolumn{1}{c}{}
\hline
$\alpha=0.01$&\multicolumn{3}{c||}{$M_1=M_2=1$}&\multicolumn{3}{c||}{$M_1=10, \,M_2=1$}&\multicolumn{3}{c|}{$M_1=100,\,M_2=1$}\\
\hline
$N$&10&100&1000&10&100&1000&10&100&1000\\\hline
$R_{\rm W}$&5.586&5.362  & 5.340 &2.666  & 2.644   & 2.642 & 2.374 & 2.372 & 2.372
\\
\hline
\end{tabular}
\\~\\
\begin{tabular}{|c|c|c|c||c|c|c||c|c|c|}
%\cline{2-10}
%\multicolumn{1}{c}{}
\hline
$\alpha=0.001$&\multicolumn{3}{c||}{$M_1=M_2=1$}&\multicolumn{3}{c||}{$M_1=10, \,M_2=1$}&
\multicolumn{3}{c|}{$M_1=100,\,M_2=1$}\\
\hline
$N$&10&100&1000&10&100&1000&10&100&1000\\\hline
$R_{\rm W}$&4.908& 4.689& 4.667&2.368 &2.346 & 2.344&2.114&2.112&2.112
\\
\hline
\end{tabular}
\caption{Ratios $R_{\rm W}$ of the widths of the confidence intervals as given by Corollary \ref{vinc} and their lower bounds
from Proposition \ref{pro:lower2arik}.
}
\label{arik_exp}
\end{table}

\end{enumerate}

\subsection{Constrained case} \label{ccase}
Now consider a convex stochastic problem of the form
\begin{equation}\label{CSP}
\Opt=\min_{x\in X}\left[f_0(x):=\int_\Xi F_0(x,\xi)P(d\xi):\;\;f_i(x):=\int_\Xi F_i(x,\xi) P(d\xi)\leq 0,\,1\leq i\leq m\right],
\end{equation}
where, similarly to the above, $X$ is a convex compact set in a Euclidean space $E$, $P$ is a Borel probability distribution on  $\bR^k$, $\Xi$ is the support of $P$, and
$$
F_i(x,\xi):E\times\Xi\to \bR,\, 0\leq i\leq m,
$$
are Borel functions convex in $x$ and $P$-summable in $\xi$ for every $x$, implying that the functions $f_i$, $0\leq i\leq m$, are convex. As in the previous section, we assume that $E$ is equipped with a norm $\|\cdot\|$, the conjugate norm being $\|\cdot\|_*$, and a compatible with $\|\cdot\|$ distance-generating function  for the unit ball $B_{\|\cdot\|}$ of the norm. 

We put
\[
L(x,\xi)=\max\limits_{0\leq i\leq m}\left\{\|g-h\|_{*}: g\in \partial_x F_i(x,\xi),h\in\partial f_i(x)\right\}.
\]
Assuming {that} we can sample from the distribution $P$, and given a sample size $N$, we can build  Sample Average Approximations (SAA's) of functions $f_i$, $0\leq i\leq m$:
$$
f_{i,N}(x,\xi^N)={1\over N} \sum_{t=1}^NF_i(x,\xi_t).
$$
Here, as above, $\xi_1,\xi_2,...$ are drawn, independently of each other from $P$ and $\xi^N=(\xi_1,...,\xi_N)$. Same as above, we want to use these SAA's of the objective and the constraints of (\ref{CSP}) to infer conclusions on the optimal value of the problem of interest (\ref{CSP}).
\par
Our first observation is that {\sl in the constrained case, one can hardly expect a reliable and tight approximation to $\Opt$ to be obtainable from noisy information.} The reason is that {\sl in the general constrained case, even the special one where $F_i$ (and thus $f_i$) are affine in $x$, the optimal value is highly unstable: arbitrarily small perturbations of the data} (e.g., the coefficients of affine functions $F_i$ in the special case
or parameters of distribution $P$) {\sl can result in large changes in the optimal value.} As a result, with noisy observations of the data, one  could hardly expect to get a good estimate of $\Opt$ via a sample of instance-independent size. The standard remedy is to impose
an priori upper bound on the magnitude of optimal Lagrange multipliers for the problem of interest, e.g., by imposing the assumption that this problem is strictly feasible, with the level of strict feasibility\\
 \begin{equation}\label{defbeta}
\varkappa:=-\min_{x\in X}\max[f_1(x),...,f_m(x)]
  \end{equation}
 lower-bounded by a known in advance positive quantity. Since in many cases an priori lower bound on $\varkappa$ is unavailable,  we intend in the sequel to utilize an alternative approach, specifically, as follows. Let us associate with (\ref{CSP}) the univariate (max-)function
$$
\Phi(r)=\min_{x\in X}\max[f_0(x)-r,f_1(x),...,f_m(x)].
$$
Clearly, $\Phi$ is a continuous convex nonincreasing function of $r\in \bR$ such that $\Phi(r)\to\infty$ as $r\to-\infty$. This function has a zero if and only if (\ref{CSP}) is
feasible, and $\Opt$ is nothing but the smallest zero of $\Phi$.
\begin{definition}
Given $\epsilon>0$, a real $\rho$ {\sl  $\epsilon$-underestimates $\Opt$} if $\rho\leq\Opt$ and $\Phi(\rho)\leq\epsilon$.
\end{definition}
Note that  $\Phi(\rho)\leq\epsilon$ implies that
$$
\rho\geq\Opt(\epsilon):=\min_{x\in X}\left[f_0(x)-\epsilon:f_i(x)\leq\epsilon,\,1\leq i\leq m\right].
$$
Thus, $\rho$ $\epsilon$-underestimates $\Opt$ if and only if $\rho$ is in-between the optimal value of the problem of interest (\ref{CSP}) and the problem obtained from (\ref{CSP}) by ``optimistic'' $\epsilon$-perturbation of the objective and the constraints.

\begin{remark}
Let $\rho$ $\epsilon$-underestimate $\Opt$. When (\ref{CSP}) is feasible and the magnitude (absolute value)
$\vartheta$ of the left derivative of $\Phi(\cdot)$ taken at $\Opt$ is positive, from convexity of $\Phi$ it follows that
$$
\Opt-{\epsilon\over \vartheta}\leq\rho<\Opt.
$$
\end{remark}
Thus, unless $\vartheta$ is small, $\rho$ is an $O(\epsilon)$-tight lower bound on $\Opt$. Note that when (\ref{CSP}) is strictly feasible, $\vartheta$ indeed is positive, and it can be bounded away from zero. Indeed, we have the following:
\begin{lemma}\label{vincentlemma}
Let $\vartheta$ be the magnitude of the left derivative of $\Phi$ at $\Opt$ and assume that $\varkappa$ given by  \eqref{defbeta} is positive.
Then
$$
\vartheta \geq {\varkappa\over V+\varkappa},\;\mbox{where}\;V=\max_{x\in X} f_0(x)-\Opt.
$$
\end{lemma}
%\end{remark}
In respect to the constrained problem (\ref{CSP}), our main result is as follows:
\begin{proposition}\label{propmain1}
In the just described situation, assume that $f_i$, $0\leq i\leq m$, are differentiable on $X$, and that for some positive $M_1$, $M_2$ one has for $i=0,1,...,m$ and all $x\in X$:
\[
\begin{array}{ll}
\bE\Big[\e^{(F_i(x,\xi)-f_i(x))^2/M_1^2} \Big] \leq \e,\quad &
\bE\Big[\e^{L^2(x,\xi)/M_2^2} \Big] \leq \e.
\end{array}
\]
Assume also that {\rm (\ref{CSP})} is feasible, and that for $N\in \bZ_+$, $s>1$,
and $\lambda,\,\mu\in[0,2\sqrt{\tau_*N}]$, $\epsilon$ and $\beta$ satisfy
\begin{equation}\label{epsilon}
\begin{array}{rcl}
\epsilon&>&2N^{-1/2}\left[\mu M_1+M_2 R \left[{\Omega\over 2}[1+s^2]+\lambda\right]\right],\\
\beta&=&\beta(\mu,s,\lambda,N)=  \e^{-N(s^2-1)}+\e^{-{\lambda^2\over 4\tau_*}} +(m+2)\e^{-{\mu^2\over 4\tau_*}},
\end{array}
\end{equation}
where $\Omega$ is given by {\rm \rf{Omega}}, and $\tau_*$ is given in Proposition \ref{propmain}. Then
the random quantity
\[%\begin{equation}\label{OptN}
\Opt_N(\xi^N)=\min\limits_{x\in X}\Big[
%\underbrace{
f_{0,N}(x,\xi^N)-\mu M_1 N^{-1/2}:\;
%}_{\bar{f}_{0,N}(x,\xi^N)}: \underbrace{
f_{i,N}(x,\xi^N)-\mu M_1 N^{-1/2}
%}_{\bar{f}_{i,N}(x,\xi^N)}
\leq0,\,1\leq i\leq m\Big]
\]%\end{equation}
$\epsilon$-underestimates $\Opt$ with probability $\geq1-\beta$.\\
\end{proposition}
\par {\textbf{MinMax Stochastic Optimization.}} The proof of Proposition \ref{propmain1} also yields the following result which is of interest by its own right:
\begin{proposition}\label{corpropmain1} In the notation and under assumptions of Proposition \ref{propmain1}, consider the minimax problem
\begin{equation}\label{MinMax}
\Opt=\min_{x\in X}\max[f_1(x),...,f_m(x)]
\end{equation}
along with its {Sample} Average Approximation
$$
\Opt_N(\xi^N)=\min_{x\in X} \max[f_{1,N}(x,\xi^N),...,f_{m,N}(x,\xi^N)].
$$
Then for every $N\in\bZ_+$, $s>1$ and $\lambda,\mu\in [0,2\sqrt{\tau_*N}]$ one has
\begin{equation}\label{again}
\Prob\left\{\Opt_N(\xi^N)>\Opt+\mu M_1 N^{-1/2}\right\}\leq
m\e^{-{\mu^2\over 4\tau_*}}
\end{equation}
and
\begin{equation}\label{onceagain}
\begin{array}{l}
\Prob\left\{\Opt_N(\xi^N)<\Opt-\left[\mu M_1+2M_2\left[{\Omega\over 2}[1+s^2]+2\lambda\right]\right]N^{-1/2}\right\}\\
~~\leq
\e^{-{\mu^2\over 4\tau_*}}+2\left[\e^{-N(s^2-1)}+\e^{-{\lambda^2\over 4\tau_*}}\right].
\end{array}\end{equation}
\end{proposition}
An attractive feature of  bounds \rf{again} and \rf{onceagain} is that they are only weakly affected by the number $m$ of components in the minimax problem (\ref{MinMax}).
\section{Numerical experiments}\label{sec:num}
The goal of the experiments of this section is to illustrate numerically the ideas developed above.

\subsection{Confidence intervals for problems without stochastic constraints}\label{sec:conf_ass}
Here we consider three risk{-averse} optimization problems of the form \eqref{eq:opt1} and
we compare the properties of three confidence intervals for $\Opt$ computed for the confidence level $1-\alpha=0.9$:
\begin{enumerate}
\item the asymptotic confidence interval
\begin{equation}\label{confintervalas}
{\cal C}_a(\alpha) =\left[ \widehat{f}_{2N}  -q_{\cN}\left(1-\frac{\alpha}{2}\right){\widehat{\sigma}_{2N}\over \sqrt{{N}}},\,
\widehat{f}_{2N}   +q_{\cN}\left(1-\frac{\alpha}{2}\right){\widehat{\sigma}_{2N}\over \sqrt{N}}\right].
\end{equation}

 Here $\widehat{f}_{2N}$ and $\widehat{\sigma}_{2N}^2$ are estimations of expectation $f(x(\xi^N))$ of $F(x(\xi^N),\xi')$ and of its variance, taken over the distribution of  independent from $\xi^N$ 
 random vector $\xi'$ (here $x(\xi^N)$ is the SAA \eqref{SAAdef} optimal solution  built using the $N$-sample $\xi^N$). They are computed using a second sample $\bar{\xi}^N$ of $\xi$ of size $N$ independent of $\xi^N$:
\be
{\widehat{f}_{2N}}={1\over N}\sum_{t=1}^N F(x(\xi^{N}),\bar{\xi}_t),\;\;\; \widehat{\sigma}_{2N}^2={1\over N}\sum_{t=1}^N F(x(\xi^N),\bar{\xi}_t)^2-\widehat{f}^2_{2N}
\ee{Fsig} (for a justification, see \cite{shapiro2014}).
\item
The (non-asymptotic) confidence interval ${\cal C}_{\rm SMD}(\alpha)$ is built using the {{\em{offline}}} accuracy certificates for the Stochastic Mirror Descent algorithm, cf. Section 3.2 and 
Theorem 2 of  \cite{lan2012validation}.
The non-Euclidean algorithm with entropy distance-generating function provided the best results in these experiments and was used for comparison.
\item
The (non-asymptotic) confidence interval, denoted ${\cal C}_{\rm SAA}(\alpha)$, is based on the bounds of Proposition \ref{propmain}. Specifically, we use
the lower  $1-\alpha/2$-confidence bound ${\tt{Low}}^{\tt{SAA}}$ of Corollary \ref{vinc}.
To construct the upper bound we proceed as follows: first we compute the optimal solution $x(\xi^{N})$ of the SAA using a simulation sample $\xi^N$ of size $N$; then we compute an estimation 
$\widehat{f}_{2N}$ of the objective value
 using the independent sample $\bar{\xi}^N$  as in \rf{Fsig}. Finally, we build the upper confidence bound
\[{\tt Up}'=\widehat{f}_{2N}+{2M_1\sqrt{\tau^*\ln [4\alpha^{-1}]\over N}},\]
where $\tau^*$ and $M_1$ are as in Proposition \ref{propmain} (cf. the bound  \rf{upper}). Finally, the upper bound
 $\overline{\tt Up}^{\tt SAA}$ computed as the minimum of ${\tt Up}'$ and the upper bound ${\tt Up}^{\tt SAA}$ by 
 Corollary \ref{vinc}, tuned for the confidence level $1-\alpha/4$, was used.\footnote{It is worth to mention that in our experiments the upper bound ${\tt Up}^{\tt SAA}$ was too conservative and was systematically ``outperformed'' by the upper bound ${\tt Up}'$.}
\end{enumerate}
For the sake of completeness, for the three optimization problems considered in this section we provide the detail of computing of the constants involved in Appendix \ref{B-appendix}.
SAA formulations of these problems were solved numerically using Mosek Optimization Toolbox \cite{mosek}.

\subsubsection{Quadratic risk minimization} \label{stoquadopt}

Consider the following instance of problem \eqref{eq:opt1}: let $X$ be the standard simplex in $\bR^n$: $X=\{x\in\bR^{n}:\; x_i\geq0,\sum_{i=1}^{n} x_i=1\}$, $\Xi$  is a part of the unit box
$\{\xi=[\xi_1;...;\xi_n]\in\bR^n:\|\xi\|_\infty\leq 1\}$,
\[
F(x,\xi)=\kappa_0 \xi^T x  + {\kappa_1 \over 2} \left( {\xi^T x} \right)^2,\;\;
f(x)=\kappa_0 \mu^Tx +{\kappa_1 \over 2} x^TVx,
\]
with $\kappa_1  \geq 0$ and $\mu=\bE\{\xi\}$,  $V=\bE\{\xi\xi^T\}$.

In our experiments, $\kappa_0 =0.1,\, \kappa_1 =0.9$, and
$\xi$ has independent Bernoulli entries: $\Prob(\xi_i=1)=\theta_i, \;\Prob(\xi_i=-1)=1-\theta_i$, with $\theta_i$ drawn uniformly over $[0,1]$.
This implies that
\[
\mu_i=2\theta_i - 1,\;\;V_{i, j}=\left\{
\begin{array}{ll}\bE\{\xi_i \}\bE\{\xi_j \}=(2\theta_i - 1)(2\theta_j - 1)\;&\mbox{for $i \neq j$},\\
\bE\{ \xi_i^2 \}=1\;&\mbox{for $i=j$.}\end{array}\right.
 \]
{
For several problem and sample sizes, we present in Table \ref{covprobaex1} the  empirical ``coverage probabilities'' of the ``asymptotic'' confidence interval ${\cal C}_a(\alpha)$ 
(i.e., the ratio of realizations for which ${\cal C}_a(\alpha)$ covers the true optimal value) for $\alpha=0.1$ and ``target
coverage probability'' $1-\alpha =0.9$, computed over 500 realizations (the coverage probabilities of the two non-asymptotic confidence intervals are equal to one for all parameter combinations).
We observe that empirical coverage probabilities degrade when the problem size $n$ increases
(and, as expected, they tend to increase with the sample size).
For instance,  these probabilities
are much smaller than the target level, unless the size $N$ of the simulation sample is much larger than problem dimension $n$.
\begin{table}
\centering
\begin{tabular}{|c||c|c|c|c|}
\hline
Sample& \multicolumn{4}{c|}{Problem size $n$}\\
\cline{2-5}
size $N$ &2&10&20&100\\
\hline
$20$&0.94& 0.68 &  0.59 &0.10 \\
\hline
$100$& 0.95& 0.87& 0.70& 0.46      \\
\hline
$10\,000$ & 0.94 & 0.95 & 0.91 & 0.85       \\
\hline
\end{tabular}
\caption{Quadratic risk minimization. Estimated coverage probabilities of the asymptotic confidence intervals ${\cal C}_a(0.1)$.}
\label{covprobaex1}
\end{table}
On the other hand, not surprisingly, the non-asymptotic bounds yield confidence intervals much larger than the asymptotic confidence interval. We report in Table \ref{ratioex1table1} the mean ratio of the widths of
non-asymptotic -- ${\cal C}_{\rm SAA}(\alpha)$ and ${\cal C}_{\rm SMD}(\alpha)$ -- and asymptotic confidence intervals ${\cal C}_{a}(\alpha)$.\footnote{Note that asymptotic estimation $\widehat{\sigma}_N$ of the noise variance often degenerates. 
To avoid division by zero problems, we only kept the realizations where asymptotic confidence intervals cover the true optimal value.} These ratios increase significantly with problem size (in part because the
asymptotic interval becomes indeed too short), and we observe that the confidence interval ${\cal C}_{\rm SAA}(\alpha)$ based on
Sample Average Approximation remains much smaller than the interval  ${\cal C}_{\rm SMD}(\alpha)$ yielded by Stochastic Approximation.
\begin{table}
\centering
\begin{tabular}{|c||c|c|c|c|c||c|c|c|c|c|}
\hline
Sample& \multicolumn{5}{c||}{${|{\cal C}_{\rm SAA}(\alpha)|\over |{\cal C}_{a}(\alpha)|}$, problem size $n$ }& \multicolumn{5}{c|}{${|{\cal C}_{\rm SMD}(\alpha)|\over |{\cal C}_{a}(\alpha)|}$, problem size $n$ }\\
\cline{2-6}\cline{7-11}
size $N$ &2&10&20&100&200&2&10&20&100&200\\
\hline
100 & 6.37 & 9.18  & 10.18 & 29.50&47.43 &30.57  &   65.87 & 78.5   & 274.63  &474.68\\
\hline
1000 &  3.27 &  4.33&   4.52&13.92  &22.46 & 15.52   & 32.56   & 36.98   & 134.67   & 232.32 \\
\hline
10\,000&  3.15 & 4.37&  4.40&  13.44&  21.96& 15.46  & 32.40  & 35.87  & 131.70 &227.56\\
\hline
\end{tabular}
\caption{Quadratic risk minimization. Average ratio of the widths of the
non-asymptotic and asymptotic confidence intervals.}
\label{ratioex1table1}
\end{table}

\corri{Finally, on Figure \ref{vfig1} we compare average  over 100 problem realizations ``inaccuracies'' of approximate solutions delivered by SAA and SMD for ``typical'' problem instances of size $n=100$  
for two combinations of parameters $\kappa_0$ and 
$\kappa_1 $.
\begin{figure}

\begin{tabular}{cc}
\includegraphics[scale=0.55]{LogLog_Sto_Quad_1.pdf}&\includegraphics[scale=0.55]{LogLog_Sto_Quad_3_A101.pdf}
\end{tabular}
\caption{Quadratic risk minimization:  empirical estimation of $\mathbb{E}\{f(x^N)\} -\Opt$ over $100$ realizations as a function of $N$ with approximate solution $x^N$ obtained using either
SAA or SMD (logarithmic scale). Left plot: simulation results for a problem with $\kappa_0 =0.1, \kappa_1 =0.9$; right plot: results for a problem with $\kappa_0 =0.9, \kappa_1 =0.1$.}
\label{vfig1}
\end{figure}
}{}

\subsubsection{Gaussian VaR optimization}\label{portfolioproblem}

We consider the instance of problem \eqref{eq:opt1} where $X\subset \bR^n$ is the standard simplex,
$\xi$ has normal distribution $\cN(0,\Sigma )$ on $\bR^n$ with $\Sigma_{i, i}\leq \smax$,
and $F(x,\xi)=\kappa_0 \xi^ Tx + \kappa_1 | \xi^T x |$, with $\kappa_1 \geq0$, so that
 $f(x)=\kappa_1  \sqrt{2/\pi} \sqrt{x^T \Sigma x}$. 
 Observe that in the present situation, minimizing $f(x)$ is equivalent to maximizing the 
 $\varepsilon$-quantile of the distribution of $\xi^Tx$ (Value-at-Risk ${\rm VaR}(\varepsilon)$) with $\varepsilon=1-\Psi\left(\kappa_1  \sqrt{2/\pi}\right)$ 
 where $\Psi(\cdot)$ is the standard normal CDF.

We generated instances of the problem of different sizes with $\kappa_0 =0.9$, $\kappa_1  = 0.1$, and diagonal matrix $\Sigma$ with diagonal entries drawn uniformly
over $[1,6]$ ($\smax = \sqrt{6}$).

We reproduce the experiments of the previous section in this setting, namely, for several problem and sample sizes, we compute empirical ``coverage probabilities'' of
the confidence intervals over 500 realizations. We report the results for the ``asymptotic'' confidence interval ${\cal C}_a(\alpha)$ in Table \ref{covprobaex2} for ``target
coverage probability'' $1-\alpha=0.9$ (same as above, coverage probabilities of non-asymptotic intervals are equal to one for all parameter combinations).
We especially observe extremely low coverage probabilities for $n=100$ and $N=20$ or $N=100$.
\begin{table}
\centering
\begin{tabular}{|c||c|c|c|c|}
\hline
Sample& \multicolumn{4}{c|}{Problem size $n$}\\
\cline{2-5}
size $N$ &2&10&20&100\\
\hline
$20$& 0.95&  0.73& 0.53   & 0.05\\
\hline
$100$& 0.9 & 0.78  & 0.48  & 0.006    \\
\hline
$10\,000$ & 0.92     & 0.91 &  0.92 & 0.68   \\
\hline
$100\,000$ & 0.94 & 0.92 & 0.92 & 0.92 \\
\hline
\end{tabular}
\caption{Gaussian VaR optimization. Estimated coverage probabilities of asymptotic confidence intervals.}
\label{covprobaex2}
\end{table}

In Table \ref{ratioex2table1} the average ratios of the widths of
non-asymptotic and asymptotic confidence intervals are provided for the same experiment.
{Same as in the experiments described in the  previous section, these ratios increase with problem size, and the confidence intervals by SMD are much more conservative than those by SAA.}

\begin{table}
\centering
\begin{tabular}{|c||c|c|c|c|c||c|c|c|c|c|}
\hline
Sample& \multicolumn{5}{c||}{${|{\cal C}_{\rm SAA}(\alpha)|\over |{\cal C}_{a}(\alpha)|}$ for problem size $n$}&
\multicolumn{5}{c|}{${|{\cal C}_{\rm SMD}(\alpha)|\over |{\cal C}_{a}(\alpha)|}$ for problem size $n$}\\
\cline{2-6}\cline{7-11}
size $N$ &2&10&20&100&200&2&10&20&100&200\\
\hline
20 & 4.42    & 6.15 & 6.11 & 6.27 & 6.35   &40.16   & 112.38  & 133.80  &183.61  & 205.66  \\
\hline
100 & 5.04    & 9.11& 10.79 & 12.87&  13.44 &  46.41 & 172.00  &244.68   & 397.01&  458.85 \\
\hline
10\,000& 5.27 &12.17  & 16.29  & 26.65 & 30.28 & 49.15 &237.79 & 386.31  &974.32 & 1088.90\\
%\hline
%100\,000&  8.5 & 20.1 & 28.1 &62.1 & 82.6 &  411.6  & 707.2  & 2 218.9\\
\hline
\end{tabular}
\caption{Gaussian VaR optimization. Average ratio of the widths of the
non-asymptotic and asymptotic confidence intervals.}
\label{ratioex2table1}
\end{table}
\corri{
On Figure \ref{vfig2} we present average over 100 problem realizations ``inaccuracies'' of approximate solutions delivered by SAA and SMD for ``typical'' problem instances of size $n=100$.
\begin{figure}
\centering
\begin{tabular}{cc}
 \includegraphics[scale=0.55]{Loglog_Markowitz_A1_09.pdf}
 & \includegraphics[scale=0.55]{Loglog_Markowitz_A1_01.pdf}
 \end{tabular}
\caption{Gaussian VaR optimization: empirical estimation of $\mathbb{E}\{f(x^N) -\Opt\}$ as a function of $N$ (in logarithmic scale).
Left plot: simulation results for a problem with $\kappa_0 =0.1, \kappa_1 =0.9$; right plot: results for a problem with $\kappa_0 =0.9, \kappa_1 =0.1$.}
\label{vfig2}
\end{figure}
}{}
\subsubsection{CVaR optimization} \label{portboundedreturns}

We consider here the following CVaR optimization problem:
given $\varepsilon>0$, find
 \be
 \begin{array}{rl}
\Opt_\varepsilon=\min_{x'}& \kappa_0 \bE\{\xi^Tx'\}+\kappa_1 {\rm CVaR}_\varepsilon(\xi^Tx') \\
&x'\in \bR^n\;\sum_{i=1}^n x'_i=1,\; x'\ge 0,
\end{array}
 \ee{cvar}
 where the support $\Xi$ of $\xi$ is a part of the unit box
$\{\xi=[\xi_1;...;\xi_n]\in\bR^n:\|\xi\|_\infty\leq 1\}$,
 and where
 \[
 {\rm CVaR}_\varepsilon( \xi^Tx')=\min_{x_0 \in \bR}\, \{x_0 + \bE\{\varepsilon^{-1}[\xi^Tx'-x_0]_+\} \}
 \]
is the Conditional Value-at-Risk of level $0<\varepsilon<1$, see \cite{ury2}.
Observing that $|\xi^Tx'| \leq 1$ a.s., the above problem is clearly of the form \eqref{eq:opt1} with
$X=\{x=[x_0;x'_1;...;x'_n]\in\bR^{n+1}: \;|x_0|\leq 1,\, x'_1,...,x'_{n}\geq0,\,\sum_{i=1}^{n} x'_i=1\}$
and
\[
F(x,\xi)=\kappa_0 \xi^Tx' + \kappa_1 \left(x_0+{1\over\epsilon}[\xi^Tx'-x_0]_+\right).
\]
We consider random instances of the problem with $\kappa_0,\kappa_1 \in[0,1]$, and
$\xi$  with independent Bernoulli entries: $\Prob(\xi_i=1)=\theta_i, \;\Prob(\xi_i=-1)=1-\theta_i$, with $\theta_i$, $i=1,...,n$ drawn uniformly from $[0,1]$.

 We compare the non-asymptotic confidence interval $\C_{\rm SAA}(\alpha)$ for $\Opt_\varepsilon$ to the asymptotic confidence interval $\C_a(\alpha)$ with confidence level $1-\alpha=0.9$. We consider two sets
 of problem parameters: $(\kappa_0, \kappa_1 , \varepsilon)=(0.9, 0.1, 0.9)$ and the risk-averse variant
$(\kappa_0, \kappa_1 , \varepsilon)=(0.1, 0.9, 0.1)$.
The empirical coverage probabilities
for the asymptotic confidence interval are reported in Table \ref{covprobaex3}.
As in other experiments, the coverage probability  is still below the target probability $1-\alpha=0.9$ when
the sample size is not much larger than the problem size. For SAA, the coverage probabilities
are equal to one for all parameter combinations.
\begin{table}
\centering
\begin{tabular}{|c||c|c|c|c||c|c|c|c|}
\hline
Sample& \multicolumn{4}{c||}{$\varepsilon=0.1$, problem size $n+1$}&\multicolumn{4}{c|}{$\varepsilon=0.9$, problem size $n+1$}\\
\cline{2-5}\cline{6-9}
size $N$ &3&11&21&101&3&11&21&101\\
\hline
$100$& 0.96 & 0.74  & 0.85   & 0.78&  0.96   &0.95 &0.95&0.78\\
\hline
$1000$&   0.95  & 0.88  & 0.86 &  0.67& 0.95 & 0.92 & 0.84& 0.84\\
\hline
$10\,000$ & 0.92     & 0.93 &  0.91& 0.94 &0.92 & 0.95 & 0.96 & 0.96  \\
\hline
\end{tabular}
\caption{CVaR optimization. Estimated coverage probabilities of asymptotic confidence intervals.}
\label{covprobaex3}
\end{table}

We report in Table \ref{ratioex3table1} the average ratio of the widths of
non-asymptotic and asymptotic confidence intervals. Note that the Lipschitz constant of $F(\cdot,\xi)$ is proportional to $1/\varepsilon$ when $\varepsilon$ is small. This explains the fact that for small values of $\epsilon$, the ratio
of the widths of the proposed non-asymptotic and asymptotic confidence intervals grows up significantly, especially for problem size $n+1=3$.

The experiments of this section show that when the sample size is not much larger than the problem dimension, the asymptotic computations fail
to provide the confidence set of the prescribed risk. In such case the proposed approach, though conservative, seems to be the only option available
for constructing a reliable confidence interval.

\begin{table}
\centering
\begin{tabular}{|c||c|c|cc|c||c|c|c|c|c|}
\hline
Sample& \multicolumn{5}{c||}{$\varepsilon=0.9$,  problem size $n+1$}&
\multicolumn{5}{c|}{$\varepsilon=0.1$,  problem size $n+1$}\\
\cline{2-11}
size $N$ &3&11&21&101&201&3&11&21&101&201\\
\hline
100 &3.09    & 3.69& 7.33 & 14.25& 13.79 &293.47    &27.61 &  9.14&14.32 &  14.44\\
\hline
1000&3.25   &3.67 & 8.63 &35.04 & 36.72&294.16   &27.04 &8.72  & 34.43  & 37.42 \\
\hline
10\,000& 3.22& 3.68&8.61  &32.08 &34.00 & 293.92  &26.91 & 8.66 & 31.70&34.18\\
\hline
\end{tabular}
\caption{CVaR optimization. Average ratio ${|{\cal C}_{\rm SAA}(\alpha)|\over |{\cal C}_{a}(\alpha)|}$ of the widths of the
non-asymptotic and asymptotic confidence intervals.}
\label{ratioex3table1}
\end{table}
\corri{
\begin{figure}
\centering
\begin{tabular}{c}
 \includegraphics[scale=0.6]{Loglog_CVaR_A0_09.pdf}
\end{tabular}
\caption{CVaR \corri{minimization}{optimization}: empirical estimation of $\mathbb{E}\{f(x^N) -\Opt\}$ as a function of $N$ (in logarithmic scale) on a typical problem instance with $\kappa_0=0.9, \kappa_1 =0.1$, and $\varepsilon=0.1$}
\label{vfig3}
\end{figure}
}{}

\subsection{Lower bounding the optimal value of a minimax problem}
We illustrate here the application of Proposition \ref{corpropmain1} to lower bounding the optimal value of the MinMax problem \rf{MinMax}. To this end we consider the toy problem
\be
\Opt= \min_{x}\max\left[f_i(x), \;i=1,...,3, \;x=[u;v],\;v\in \bR,\,u\in \bR^n, \,\sum_{i=1}^n u_i=1,\; u\ge 0\right],
\ee{minmax-opt}
where
\[
f_1(x)=v+\bE\{\varepsilon^{-1}[\xi^Tu-v]_+\}+\chi_1,\;\; f_2(x)= \bE\{\xi^Tu\}+\chi_2,\;\;f_3(x)= \chi_3-\bE\{\xi^Tu\},
\]
$\varepsilon$ and $\chi$ being some given parameters. The SAA of the problem reads
\be
\Opt(\xi^N)= \min_{x}\max\left[f_{i,N}(x, \xi^N), \;i=1,...,3, \;x=[u;v]\;v\in \bR,\,u\in \bR^n, \,\sum_{i=1}^n u_i=1,\; u\ge 0\right],
\ee{saa211}
with
\bse
f_{1,N}(x,\xi^N)&=&v+{1\over N\varepsilon}\sum_{t=1}^N[\xi_t^Tu-v]_++\chi_1,\\
f_{2,N}(x,\xi^N)&=& {1\over N}\sum_{t=1}^N\xi_t^Tu+\chi_2,\;\;f_{3,N}(x, \xi^N)= \chi_3-{1\over N}\sum_{t=1}^N\xi_t^Tu.
\ese
One can try to build an ``asymptotic'' lower bound for $\Opt$ as follows ({note that here we are not concerned with the theoretical validity of this construction}): given the optimal solution $x(\xi^N)$ to 
SAA \rf{saa211} and an independent sample $\bar{\xi}^N$, compute empirical 
estimations $\widehat{f}_{i,2N}$ and $\widehat{\sigma}^2_{i,2N}$ of expectation and variance of $F_i(x(\xi^N),\xi')$, as explained 
in Section \ref{sec:conf_ass}, then compute the lower bound ``of asymptotic risk $\alpha$'' according to
\[
\underline{\Opt}(\xi^N)=\max_{i=1,...,3} \left\{   \widehat{f}_{i,2N}-q_{\cN}\left(1-{\small \alpha\over 3}\right){\widehat{\sigma}_{i,2N}\over \sqrt{N}}\right\}.
\]
On Figure \ref{fig:3} we present the simulation results for the case of $\xi\in \bR^n$ with independent Bernoulli components:
$\Prob(\xi_i=1)=\theta_i, \;\Prob(\xi_i=-1)=1-\theta_i$, with $\theta_i$ randomly drawn over $[0,1]$. Parameters $\chi_i, i=1,2,3$ are chosen in such a way 
\corri{to ensure}{} that 
$f_1$, $f_2$ and $f_3$ are equal at the minimizer of \rf{minmax-opt}. More precisely, the results of $100$ simulations of the problem with $n=2$ and $N=128$ are presented on Figure \ref{fig:3} for the value of CVaR parameter $\varepsilon=0.5$ and $\varepsilon=0.1$. Note that in this case the risk of the lower bound $\underline{\Opt}(\xi^N)$ is significantly larger than the prescribed risk $\varepsilon=0.1$ already for small problem dimension -- the ``asymptotic'' lower bound failed in $33$ of 100
realizations in the experiment with $\varepsilon=0.5$, and in $36$ of $100$ realizations in the experiment with $\varepsilon=0.1$.
\begin{figure}%[H]
$$
\begin{array}{cc}
\resizebox{.5\textwidth}{!}{
 \includegraphics{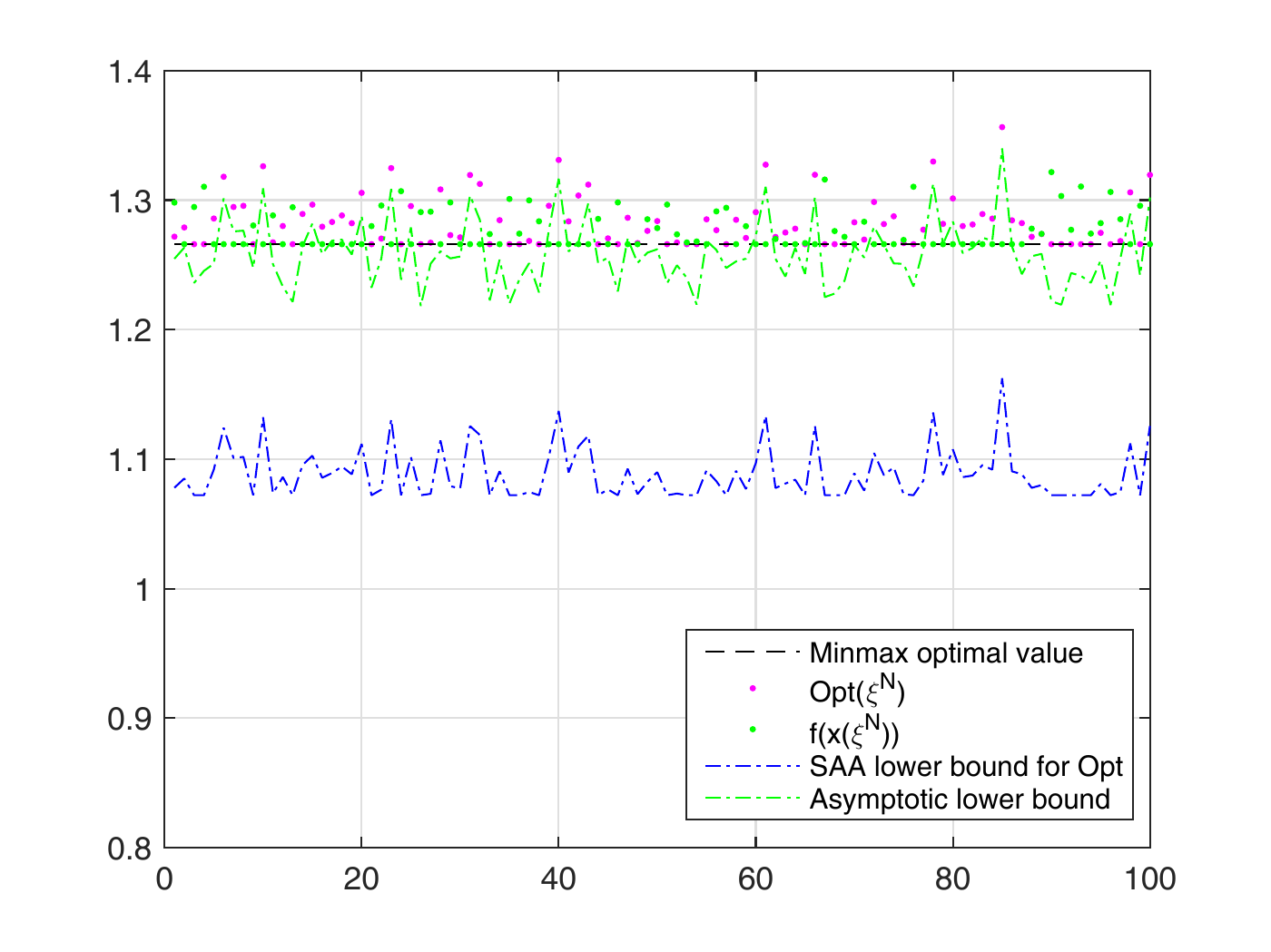}}
&
\resizebox{.5\textwidth}{!}{
 \includegraphics{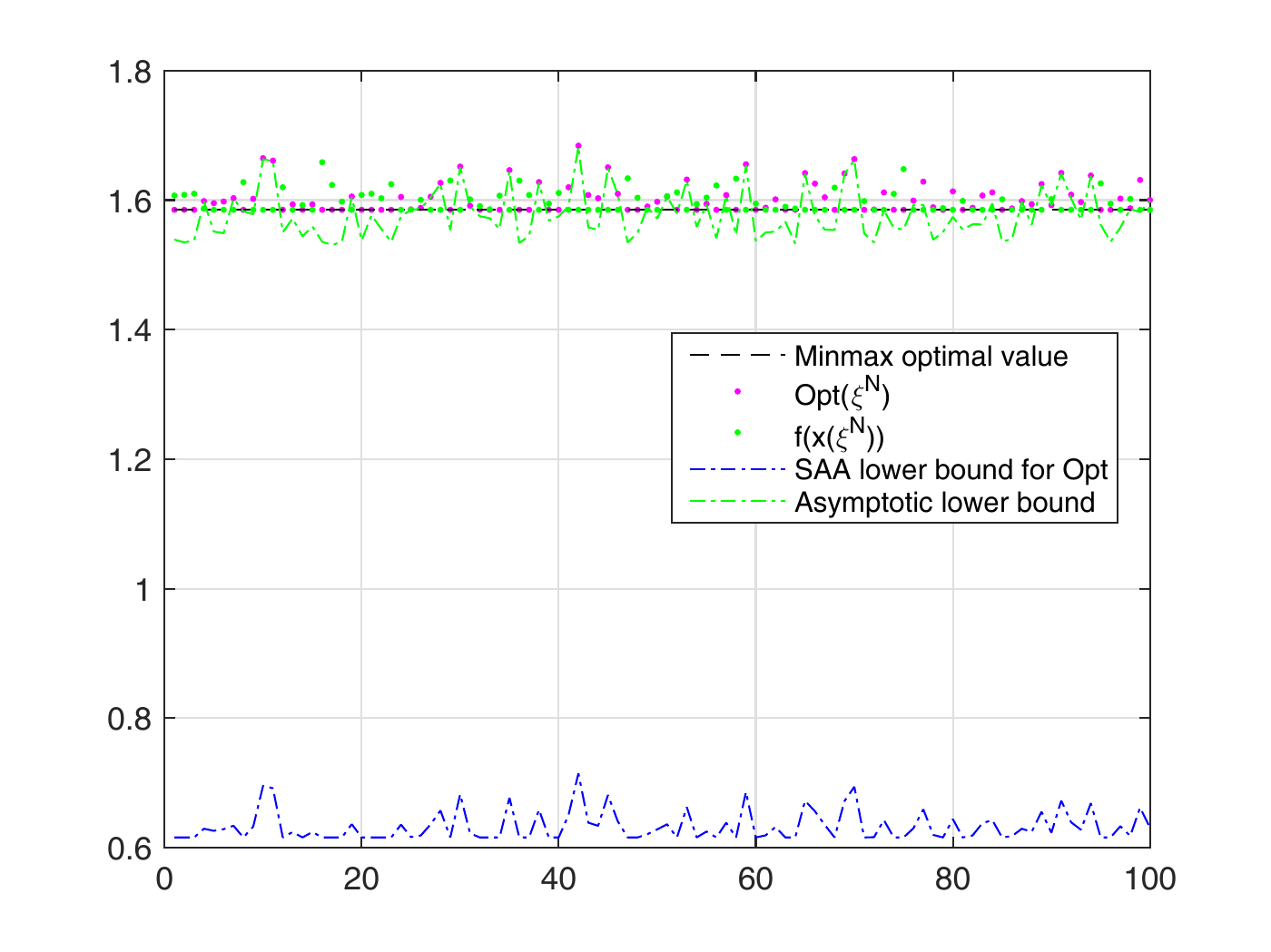}}\\
{\rm (a)} & {\rm (b)}
\end{array}
$$
\caption{ \label{fig:3}  {\small Optimal value $\Opt$ of the stochastic program \rf{minmax-opt} along with lower bound derived from the results of
{Proposition} \ref{corpropmain1} and ``asymptotic'' lower bound $\underline{\Opt}(\xi^N)$. The results for $\varepsilon=0.5$ on plot (a), for $\varepsilon=0.1$ on plot (b).
}}
\end{figure}
\subsection{Optimal value of a stochastically constrained problem}
An SAA of a stochastically constrained problem, even with a single linear constraint, can easily become unstable when the constraint is ``stiff''. As a simple illustration, let us consider a stochastically (linearly) constrained problem
\be
\Opt_\chi= \min_{x}\left[f_0(x):\;\;f_1(x)\leq 0,\;x=[u;v],\;v\in \bR,\,u\in \bR^n, \,\sum_{i=1}^n u_i=1,\; u\ge 0\right],
\ee{CVAR-Opt}
where
\[
f_0(x)=v+\bE\{\varepsilon^{-1}[\xi^Tu-v]_+\}\;\; \mbox{and}\;\;f_1(x)= \chi-\bE\{\xi^Tu\},
\]
and $\varepsilon$ and $\chi$ are problem parameters.  The SAA of the problem is
\be
\Opt_\chi(\xi^N)= \min_{x=[u;v]}\left[f_{0,N}(x):\;\;f_{1,N}(x)\leq 0,\;v\in \bR,\,u\in \bR^n, \,\sum_{i=1}^n u_i=1,\; u\ge 0\right],
\ee{saa222}
where\[
f_{0,N}(x,\xi^N  )=v+{1\over N\varepsilon}\sum_{t=1}^N[\xi_t^Tu-v]_+\;\; \mbox{and}\;\;f_{1,N}(x,\xi^N )= \chi-{1\over N}\sum_{t=1}^N \xi_t^T u.
\]
Consider now a toy example of the problem with $u\in \bR^2$, $\xi\sim\N(\mu, \Sigma)$ with $\mu=[0.1;0.5]$ and $\Sigma=\diag([1;4])$.
Let $N=128$, $\chi=0.3$, and $\varepsilon=0.1$. One can expect that in this case the optimal value $\Opt_\chi(\xi^N)$ of the SAA is unstable (in fact, problem \rf{saa222} is infeasible with probability $\Prob\left\{{1\over N}\sum_{t=1}^N\xi_{t,2}<\chi\right\}=\Prob\left\{{2\cN(0,1)\over \sqrt{N}}\leq -0.2\right\}=0.128...$). We compare the solution to \rf{saa222} with the SAA in which the 
right-hand side $\chi$ of the stochastic constraint is replaced with $\chi-\delta$ where  $\delta=q_\N(1-\varepsilon/n){\smax\over \sqrt{N}}=0.5815...$,
$\smax=\max_i\Sigma_{i, i}$.
On Figure \ref{fig:2} we present the simulation results of $100$ independent realizations of the above problem. As expected, the SAA \rf{saa222} is unstable;  the problem turned infeasible in 22\% of realizations. The SAA  with the relaxed constraint exhibits much better stability.
\begin{figure}%[H]
$$
\begin{array}{cc}
\resizebox{.5\textwidth}{!}{
 \includegraphics{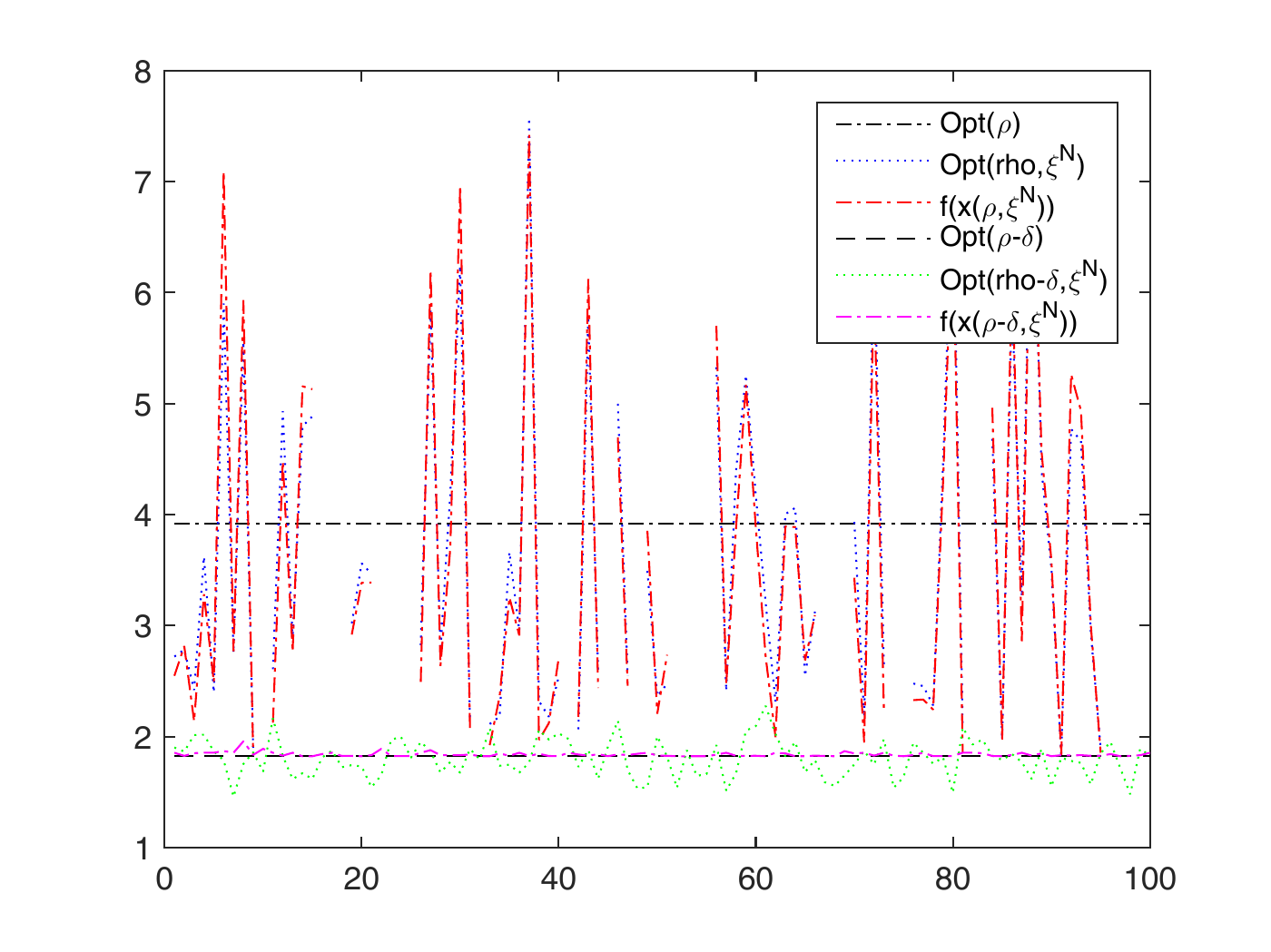}}
&
\resizebox{.5\textwidth}{!}{
 \includegraphics{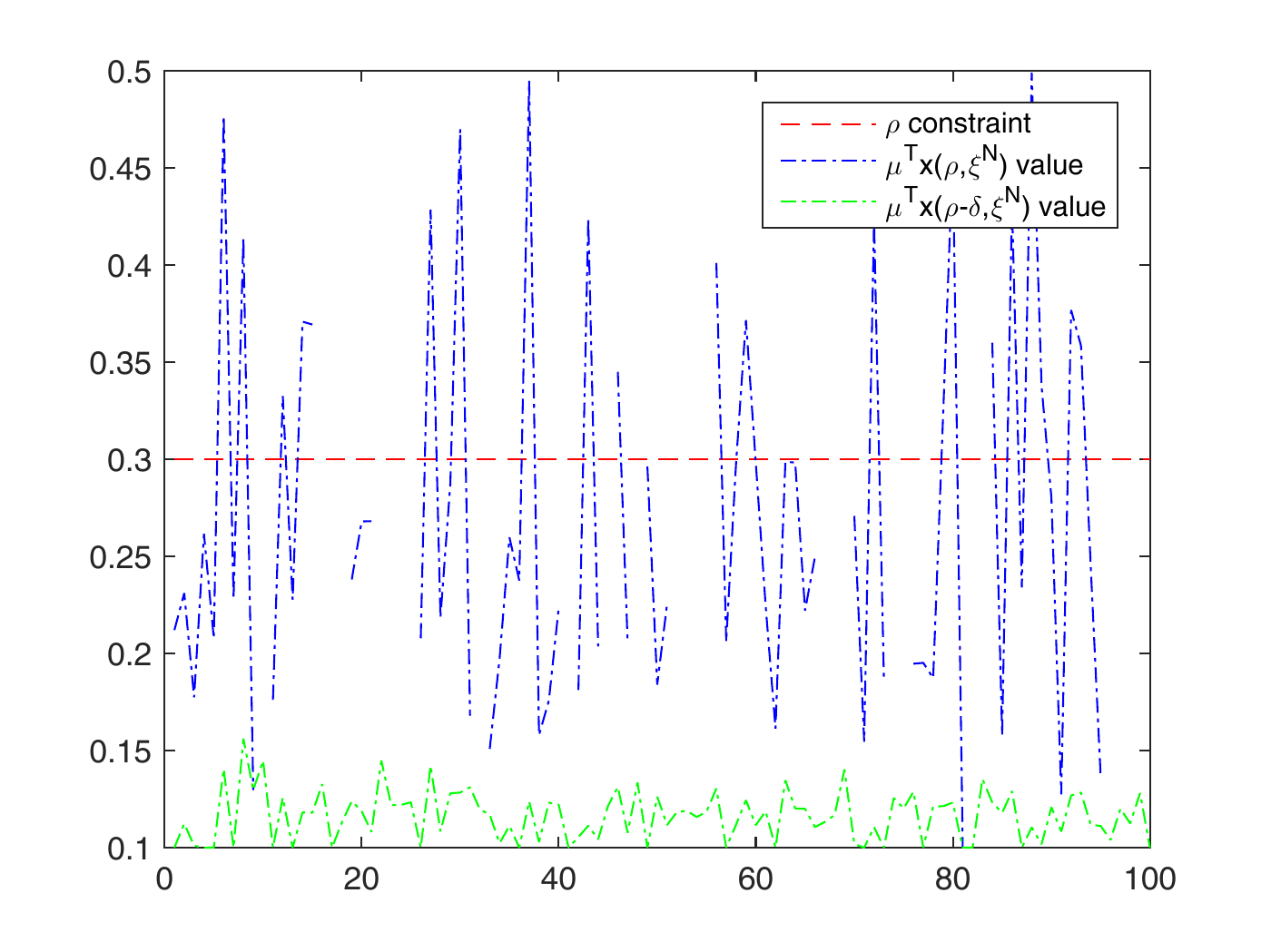}}\\
{\rm (a)} & {\rm (b)}
\end{array}
$$
\caption{ \label{fig:2}  {\small Plot (a): optimal value $\Opt$ of the stochastic program \rf{CVAR-Opt} with constraint right-hand side $\chi$ and $\chi-\delta$, along with corresponding optimal values of the SAA. Plot (b): ``true value'' of the linear form $\mu^Tx(\xi^N)$ at the SAA solution.
}}
\end{figure}

\appendix
\section{Proofs}
%\subsection{Proof of Proposition \ref{prop1}}
\subsection{Preliminaries: Large deviations of vector-valued martingales}
The result to follow is a slightly simplified and refined version of the bounds on probability of large deviations for vector-valued martingales developed in \cite{jn2008,StAppr}.
\par
Let $\|\cdot\|$ be a norm on  Euclidean space $E$, $\|\cdot\|_*$ be the conjugate norm, and $B_{\|\cdot\|}$ be the unit ball of the norm. Further, let $\omega$ be a continuously differentiable distance-generating function for $B_{\|\cdot\|}$ compatible with the norm $\|\cdot\|$ and attaining its minimum on $B_{\|\cdot\|}$ at the origin: $\omega^\prime(0)=0$, with $\omega(0)=0$ and $\Omega=\max_{x:\|x\|\leq1}\sqrt{2[\omega(x)]}$.
\begin{lemma}\label{lem1} Let $d_1,d_1,...$ be a scalar martingale-difference such that for some $\sigma>0$ it holds
$$
\bE\{\e^{d_t^2/\sigma^2}|d_1,...,d_{t-1}\}\leq\e\;\hbox{\em a.s.},\; t=1,2,...
$$

Then
\begin{equation}\label{then1}
\Prob\Big\{\underbrace{{\sum}_{t=1}^Nd_t}_{D_N}>\lambda\sigma\sqrt{N}\Big\}\leq \left\{\begin{array}{rl}
\e^{-{\lambda^2\over 4\tau_*}},& 0\leq\lambda\leq2\sqrt{\tau_*N},\\
\e^{-{\lambda^2\over 3}},& \lambda>2\sqrt{\tau_*N},
\end{array}
\right.
\end{equation}
where $\tau_*$ is defined in Proposition \ref{propmain}.
\end{lemma}
\begin{proof}
Assuming without loss of generality that $\sigma=1$ observe that
under the Lemma's premise we have $\bE\{\e^{\tau_*\theta^2d_t^2}|d_1,...,d_{t-1}\}\leq\e^{\tau_*\theta^2}$
whenever $\tau_*\theta^2\leq1$
where $\tau_*$ is defined in Proposition \ref{propmain}, and therefore for almost all $d^{t-1}=(d_1,...,d_{t-1})$ we have for $0\leq\theta\leq {1\over\sqrt{\tau_*}}$
\be
\bE\big\{\e^{\theta d_t}\big|d^{t-1}\big\}\leq \bE\big\{\theta d_t + \e^{\tau_*\theta^2d_t^2}\big|d^{t-1}\big\}
=\bE\big\{\e^{\tau_*\theta^2d_t^2}\big|d^{t-1}\big\}\leq  \e^{\tau_*\theta^2}.
\ee{thus001}
Thus, for $0\leq\theta\leq {1\over\sqrt{\tau_*}}$, we have
$\bE\{\e^{\theta D_N}\}\leq \e^{\tau_*\theta^2N}$, and $\forall \lambda>0$
\[
\Prob\{D_N>\lambda\sqrt{N}\}\leq \e^{\tau_*\theta^2N-\lambda\theta\sqrt{N}}.
\]

When minimizing the resulting probability bound over $0\leq \theta\leq {1\over\sqrt{\tau_*}}$ we get
the inequality \rf{then1} for $\lambda\in [0,2\sqrt{\tau_*N}]$: $\Prob\{D_N>\lambda\sqrt{N}\}\leq \e^{-{\lambda^2\over 4\tau_*}}$.
The corresponding bound for $\lambda>2\sqrt{\tau_*N}$ is given by exactly the
same reasoning as above in which \rf{thus001} is substituted with  the inequality
\[
\bE\left\{\e^{\theta d_t}\big|d^{t-1}\right\}\leq \bE\left\{
%\exp\left({3\theta^2\over 8}+{2d_t^2\over 3}\right)\Big|d^{t-1}
\e^{{3\theta^2\over 8}+{2d_t^2\over 3}}\Big|d^{t-1}\right\}\leq
%\exp\left({3\theta^2\over 8}+{2\over 3}\right)
\e^{{3\theta^2\over 8}+{2\over 3}}\leq  \e^{3\theta^2\over 4}
\]
when $\theta>1/\sqrt{\tau_*}$. \end{proof}

\begin{proposition}\label{prop1}
Let $(\chi_t)_{t=1,2,...}$, $\chi_t\in E$,  be a martingale-difference  such that for some $\sigma>0$ it holds
\begin{equation}\label{eq1}
\bE\Big\{\e^{\|\chi_t\|_*^2/\sigma^2}\big|\chi_1,...,\chi_{t-1}\Big\}\leq \e \;\hbox{\em a.s.},\;t=1,2,...
\end{equation}
Then for every $s>1$, we have
\begin{equation}\label{eq9}
\;\;
\Prob\left\{\|\sum_{t=1}^N\chi_t\|_*>\sigma\left[{\Omega\sqrt{N}\over 2}[1+s^2]+\lambda\sqrt{N}\right]\right\}\leq
\left\{\begin{array}{rl}\e^{-N(s^2-1)}+\e^{-{\lambda^2\over4\tau_*}},&0\leq \lambda\leq 2\sqrt{\tau_*N},\\
\e^{-N(s^2-1)}+\e^{-{\lambda^2\over 3}},&  \lambda> 2\sqrt{\tau_*N},
\end{array}\right.
\end{equation}
\end{proposition}
where  $\tau_*$ is defined in Proposition \ref{propmain} and
$\Omega$ is given by \eqref{Omega}.
\begin{proof} By homogeneity, it suffices to consider the case when $\sigma=1$, which we assume from now on.\\
\par{1$^0$.} Let $\gamma>0$. We denote
\[%\begin{equation}\label{potential}
V_x(u)=\omega(u)-\omega(x)-\langle\omega'(x),u-x\rangle\qquad\qquad[u,x\in B_{\|\cdot\|}]
\]%\end{equation}
and consider the recurrence
\[
x_1=0,\;\;x_{t+1}=\argmin_{y\in B_{\|\cdot\|}}\left[V_{x_t}(y)-\langle\gamma\chi_t ,y\rangle\right].
\]
Observe that $x_t$ is a deterministic function of $\chi^{t-1}=(\chi_1,...,\chi_{t-1})$, and that by the standard properties of proximal mapping (see. e.g. \cite[Lemma 2.1]{StAppr}),
$$
\forall (u\in B_{\|\cdot\|}): \gamma\sum_{t=1}^N\langle \chi_t,u-x_t\rangle \leq   V_{0}(u)-V_{x_{N+1}}(u)+{\gamma^2\over 2}\sum_{t=1}^N\|\chi_t\|_{*}^2
\leq \half\Omega^2+{\gamma^2\over 2}\sum_{t=1}^N\|\chi_t\|_{*}^2.
$$
Thus
$$
\max_{u\in B_{\|\cdot\|}}\Big\langle \sum_{t=1}^N\chi_t,u\Big\rangle \leq {\Omega^2\over 2\gamma}+{\gamma\over 2}\underbrace{\sum_{t=1}^N\|\chi_t\|_{*}^2}_{\eta_N}+\underbrace{\sum_{t=1}^N\langle\chi_t,x_t\rangle}_{\zeta_N}.
$$
Setting $\gamma=\Omega/\sqrt{N}$, we arrive at
\begin{equation}\label{eq2}
\max_{u\in B_{\|\cdot\|}}\Big\langle \sum_{t=1}^N\chi_t,u\Big\rangle \leq {\Omega\sqrt{N}\over 2}\left[1+{\eta_N\over N}\right]+\zeta_N.
\end{equation}
Invoking (\ref{eq1}), we get
$$
\bE\{\e^{\eta_N}\}\leq \e^N
$$
(recall that $\sigma=1$), whence
\begin{equation}\label{eq3}
\forall s>0:\;\; \Prob\big\{\eta_N>s^2 N\big\}\leq \min\left[1,\e^{N(1-s^2)}\right].\\
\end{equation}
\par{2$^0$.} When invoking (\ref{eq1}) and taking into account that $x_t$  is a deterministic function of $\chi^{t-1}$ such that $\|x_t\|\leq1$ (since $x_t\in B_{\|\cdot\|}$), we get
\begin{equation}\label{weget0}
\bE\{  \langle \chi_t,x_t\rangle |\chi^{t-1}\}=0,\;\;\bE\{ \e^{  \langle \chi_t,x_t \rangle^2 } |\chi^{t-1}\}\leq\e.
\end{equation}
%\paragraph{3$^0$.}
Applying Lemma \ref{lem1} to the random sequence $d_t= \langle\chi_t,x_t\rangle $, $t=1,2,...$ (which is legitimate, with $\sigma$ set to 1, by (\ref{weget0})), we get
\begin{equation}\label{weget00}
\Prob\left\{\zeta_N>\lambda\sqrt{N}\right\}\leq \left\{\begin{array}{rl}
\e^{-{\lambda^2\over 4\tau_*}},& 0\leq\lambda\leq2\sqrt{\tau_*N},\\
\e^{-{\lambda^2\over 3}},& \lambda>2\sqrt{\tau_*N}.
\end{array}\right.
\end{equation}
In view of (\ref{eq3}) and (\ref{weget00}), relation (\ref{eq2}) implies
the bound (\ref{eq9}) of the proposition. \end{proof}
\subsection{Proof of Proposition \protect{\ref{propmain}}}
 Let $x_*$ be an optimal solution to (SP), and let $h=\nabla f(x_*)$, so that by optimality conditions
\begin{equation}\label{h}
\langle h,x-x_*\rangle \geq0\,\,\forall x\in X.\\
\end{equation}
\par{1$^0$.} Setting $\delta(\xi)=F(x_*,\xi)-f(x_*)$, invoking (\ref{bounds}.$a$) and applying Lemma \ref{lem1} to the random sequence $d_t=\delta(\xi_t)$ and $\sigma=M_1$ (which is legitimate by (\ref{bounds}.$a$)), we get
\begin{equation}\label{weget1}
%\begin{array}{l}
\forall (N\in\bZ_+,\;\mu\in[0,2\sqrt{\tau_*N}]):\;
\Prob\Big\{{1\over N}\sum_{t=1}^N\delta(\xi_t)>\mu M_1 N^{-1/2}\Big\}\leq \e^{-{\mu^2\over 4\tau_*}}.
%\\
%\end{array}
\end{equation}
Since clearly
\[
\Opt_N(\xi^N)\leq f_N(x_*,\xi^N)=\Opt+{1\over N}\sum_{t=1}^N\delta(\xi_t),
\]
we get
\begin{equation}\label{formlowboundsaa}
\Prob\Big\{\Opt_N(\xi^N)>\Opt+ \mu M_1 N^{-1/2}\Big\}\leq \e^{-{\mu^2\over 4\tau_*}}.\\
\end{equation}

\par{2$^0$.}
It is immediately seen that under the premise of Proposition \ref{propmain}, for every measurable vector-valued function $g(\xi)\in \partial_x F(x_*,\xi)$ we have
\begin{equation}\label{hh}
h=\int_\Xi g(\xi)P(d\xi).
\end{equation}
Observe that $h_N(\xi^N)={1\over N}\sum_{t=1}^Ng(\xi_t)$ is a subgradient of $f_N(x,\xi^N)$ at the point $x_*$. Consequently, for all $x\in X$,
\bse
\lefteqn{f_N(x,\xi^N)\geq f_N(x_*,\xi^N)+\langle h_N(\xi^N),x-x_*\rangle}\\
 &\geq& \underbrace{[f(x_*)+\langle h,x-x_*\rangle]}_{\geq\Opt\hbox{\ by (\ref{h})}}
+[[f_N(x_*,\xi^N)-f(x_*)]+\langle h_N(\xi^N)-h,x-x_*\rangle]\\
&\geq&\Opt+{1\over N}\sum_{t=1}^N \delta(\xi_t) -\|h-h_N(\xi^N)\|_{*}\|x-x_*\|\geq \Opt+{1\over N}\sum_{t=1}^N \delta(\xi_t)-2\|h-h_N(\xi^N)\|_{*}R
\ese
(the concluding inequality is due to $x,x_*\in X$ and thus $\|x-x_*\|\leq 2R$ by definition of $R$). It follows that
\begin{equation}\label{eq16}
\Opt_N(\xi^N)\geq\Opt+{1\over N}\sum_{t=1}^N \delta(\xi_t)-2\|h-h_N(\xi^N)\|_{*}R.
\end{equation}
Applying Lemma \ref{lem1} to the random sequence $d_t=-\delta(\xi_t)$ we, similarly to the above, get
\begin{equation}\label{less}
\forall (N,\mu\in[0,2\sqrt{\tau_*N}]):
\Prob\left\{{1\over N}\sum_{t=1}^N \delta(\xi_t)<-\mu M_1 N^{-1/2}\right\}\leq \e^{-{\mu^2\over 4\tau_*}}.\\
\end{equation}

Further, setting $\Delta(\xi)=g(\xi)-\nabla f(x_*)$, the random vectors $\chi_t=\Delta(\xi_t)$, $t=1,2,...$, are i.i.d., zero mean (by (\ref{hh})), and satisfy the relation
$$
\bE\left\{\e^{\|\chi_t\|_{*}^2/M_2^2}\right\}\leq \e
$$
by (\ref{bounds}.$b$); besides this, $h_N(\xi^N)-h={1\over N}\sum_{t=1}^N\chi_t$. Applying Proposition \ref{prop1}, we get
$$
\begin{array}{l}
\forall (N\in \bZ_+, \,s>1,\,\lambda\in[0,2\sqrt{\tau_*N}]):\\
\Prob\{\|h-h_N(\xi^N)\|_{*}\geq M_2\left[{\Omega\over 2}[1+s^2]+\lambda\right]N^{-1/2}\}\leq \e^{-N(s^2-1)}+\e^{-{\lambda^2\over4\tau_*}}.\\
\end{array}
$$
This combines with  (\ref{eq16}), and (\ref{less}) to imply \rf{lower}.
\qed
\subsection{Proof of Proposition \protect{\ref{pro:lowersaa}}}
Due to similarity reasons, it suffices to prove the proposition for $L=R=1$. Let $B_2$ be the unit Euclidean ball of $\R^n$, and let for a unit $v\in \bR^n$ and $0<\theta\leq \pi/2$, $h_{v,\theta}$ be the spherical cap of $B_2$ with ``center'' $v$ and angle $\theta$. In other words, if $\delta=2\sin^2(\theta/2)$ is the
``elevation'' of the cap $h_{v,\theta}$ then $h_{v,\theta}=\{x\in B_2:\;v^Tx\geq 1-\delta\}$. Observe that for any $\vartheta>4\theta$ we can straightforwardly build the system $D_\theta$ of vectors in the $n$-dimensional unit sphere $S^{n-1}$ in such a way that the
angle between  every two distinct vectors of the system is $>2\theta$, so that
the spherical caps $h_{v,\theta}$ with $v\in D_\theta$ are mutually disjoint, while the spherical caps $h_{v,\vartheta}$  cover $S_{n-1}$.
If we denote by $A_{n-1}(\vartheta)$ the area of the spherical cap of angle $\vartheta\leq \pi/2$, then
 $\Card(D_\theta)A_{n-1}(\vartheta)\geq s_{n-1}(1)$, where $s_{n-1}(r)={2\pi^{n/2}r^{n-1}\over \Gamma(n/2)}$ is the area of the $n$-dimensional sphere of radius $r$. Note that  $A_{n-1}(\vartheta)$  satisfies
\[
A_{n-1}(\vartheta)=\int_0^{\vartheta} s_{n-2}(\sin t)dt=s_{n-2}(1)\int_0^{\vartheta} \sin^{n-2}t dt\leq s_{n-2}(1)\int_0^{\vartheta} t^{n-2}dt=s_{n-2}(1){\vartheta^{n-1}\over n-1}.
\]
We conclude that
\[
\Card(D_\theta) \geq{s_{n-1}(1)(n-1)\over s_{n-2}(1)\vartheta^{n-1}}\geq 3\vartheta^{1-n}
\]
for $n\geq 2$. From now on we fix $\theta=1/8$ and when choosing $\vartheta$ arbitrarily  close to $4\theta=\half$, we conclude that for any $n\geq 2$ one can build $D_\theta$ such that $\Card(D_\theta)\geq 2^n$.
\par
Now consider the following construction: for $v\in D_\theta$, let $g_{v, \theta}(\cdot):B_2\to \bR$ be defined according to
$g_{v,\theta}(x)=[v^Tx-(1-\delta)]_+$, where $\delta=2\sin^2(\theta/2)=0.0078023...$ is the elevation of $h_{v,\theta}$. Let us put
\[
f(x)=\sum_{v\in D_\theta}g_{v,\theta}(x),
 \]
 and consider the optimization problem
$
\Opt= \min [ f(x):\;x\in B_2]. $
 Since $g_{v,\delta}$ is affine on  $h_{v,\delta}$ and vanishes elsewhere on $B_2$, and $\|v\|_2=1$, we conclude that $f$ is Lipschitz continuous on $B_2$ with Lipschitz constant $1$. Let now
\[
F(x,\xi)=\sum_{v\in D_\theta} 2\xi_v g_{v,\theta}(x),
 \]
 where $\xi_{v},\;v\in D_\theta$ are i.i.d. Bernoulli random variables with $\Prob\{\xi_v=0\}=\Prob\{\xi_v=1\}=\half$. Note that $\bE_\xi\{F(x,\xi)\}=f(x)\;\forall x\in B_2$. Further, for $x\in h_{v,\theta}$,
 %\[
% \bE_\xi\{F(x,\xi)^2-f(x)^2\}=\left\{\begin{array}{ll}g^2_{v,\theta}(x)[\leq\delta^2]&x\in h_{v,\theta}\;v\in D_\theta,\\
% 0&\mbox{otherwise.}  \end{array}\right.
%\]
$
\bE_\xi\{F(x,\xi)^2-f(x)^2\}=g^2_{v,\theta}(x)\leq\delta^2,
$ and
\[
\|F'(x,\xi)-f'(x)\|^2_2=\|(2\xi_v-1) g'_{v,\theta}(x)\|_2^2\leq 1.
\]
Let us now consider the SAA $f_N(x, \xi^N)$ of $f$,
\be
f_N(x, \xi^N)={1\over N}\sum_{t=1}^N F(x,\xi_t)=\sum_{v\in D_\theta }\underbrace{{1\over N}\sum_{t=1}^N \xi_{t,v} g_{v,\theta}(x)}_{g^N_{v,\theta}(x)},
\ee{saa:121} $\xi_t,\;{t}=1,...,N$ being independent realizations of $\xi$,
and the problem of computing
\be
\Opt_N (\xi^N)=\min[ f_N(x, \xi^N):\;{x\in B_2}].
\ee{eq:32}
Note that for a given $v\in D_\theta$, $\Prob\{\sum_{t=1}^N\xi_{t,v}=0\}=2^{-N}$. Due to the independence of $\xi_v$,
we have
\[\Prob\left\{ \sum_{t=1}^N\xi_{t,v}>0,\;\forall v\in D_\theta\right\}=(1-2^{-N})^{\Card(D_\theta)}\leq  (1-2^{-N})^{2^{n}}\leq \e^{-{2^{n}\over 2^{N}}}\leq \exp(-1),
\]
for $N\leq n$. We conclude that for $N\leq n$,  with probability $\geq 1-\e^{-1}$, at least one of the summands in the right-hand side of \rf{saa:121}, let it be  $g^N_{\bar{v},\theta}(x)$, is identically zero on $B_2$. The 
optimal value $\Opt_N(\xi^N)$ of \rf{eq:32} being zero, the point $x(\xi^N)=\bar{v}$ is
clearly a minimizer of $f_N(x, \xi^N)$ on $B_2$, yet $f(x(\xi^N))=\delta$, {i.e., \eqref{bdSAA} holds with $c_0=\delta$}.\qed
\subsection{Proof of Proposition \protect{\ref{pro:lower2arik}}}
\par{1$^0$.}
Let us consider a family of stochastic optimization problems as follows. Let $\|\cdot\|=\|\cdot\|_2$ and let $X$ be the unit $\|\cdot\|_2$-ball in $\bR^n$.
Given a unit vector $h$ in $\bR^n$, positive reals $\sigma,s$ and $\delta,d$, and setting $\xi=[\eta;\zeta]\sim\cN(0,I_2)$, consider two integrands:
\[
F_0(x,\xi)=\sigma\eta h^Tx+s\zeta,\;\;\;
F_1(x,\xi)=(\delta h+\sigma\eta h)^Tx+(s\zeta-d),
\]
so that
\[
f_0(x):=\bE_{\xi}\left\{F_0(x,\xi)\right\}=0,\;\;\;
f_1(x):=\bE_{\xi}\left\{F_1(x,\xi)\right\}=\delta h^Tx-d.
\]
%Let $\gamma>0$ be given by the relation
%$$
%\bE_{\zeta\sim \cN(0,1)}\left\{\exp\{\gamma^2\zeta^2\}\right\}=\exp\{1\},
%$$
%or, equivalently,
%$$
%\gamma^2={1-\exp\{-2\}\over 2}.
%$$
Let us now check that $F_0$ and $F_1$ verify the premises of Proposition \ref{propmain}.
In the notation of Proposition \ref{propmain}, we have for $F_1$
$$
L(x,\xi)=\|[\delta h+\sigma\eta h]-\delta h\|_2=\sigma|\eta|,
$$
whence, setting $M_2=\sigma/\gamma$ with $\gamma^2=\half (1-\e^{-2})$,
$$
\bE_{\xi}\left\{\exp\{L(x,\xi)^2/M_2^2\}\right\}=\exp\{1\}.
$$
Similarly, setting $M_1=\sqrt{\sigma^2+s^2}/\gamma$, we have
$$
\bE_{\xi}\left\{\exp\{(\sigma\eta+s\zeta)^2/{M_1^2}\}\right\} = \exp\{1\},
$$
so that, for every $z\in[-1,1]$,
$$
\bE_{\xi}\left\{\exp\{(\sigma\eta z+s\zeta)^2/M_1^2\}\right\}\leq \exp\{1\}.
$$
When $x\in\bR^n$ and $\|x\|_2\leq1$, we have $F_1(x,\xi)-f_1(x)=\sigma\eta h^Tx+s\zeta$, therefore
$$
\bE_{\xi}\left\{\exp\{(F_1(x,\xi)-f_1(x))^2/M_1^2\}\right\}\leq\exp\{1\}.
$$
We conclude that $F=F_1$ satisfies the premise of Proposition \ref{propmain} with
$$
M_1=\sqrt{\sigma^2+s^2}/\gamma,\;M_2=\sigma/\gamma.
$$
It is immediately seen that $F=F_0$ satisfies the premise of Proposition \ref{propmain} with the same $M_1,\,M_2$.\\
\par{2$^0$.}
Now, with $X=\{x\in\bR^n:\|x\|_2\leq1\}$, the optimal values in the problems of minimizing over $X$ the functions $f_0$ and $f_1$ are, respectively,
$$
\Opt_0=0,\;\;\;\Opt_{1}=-\delta-d.
$$
Suppose that there exists a procedure which, under the premise of Proposition 1 with some fixed $M_1$, $M_2$, is able,
given $N$ observations of $\nabla_x F(\cdot,\xi_t),\;F(\cdot,\xi_t)$, to cover $\Opt$, with confidence $1-\alpha$, by an interval of width $W$. Note that when $W<|\Opt_1|$,
the same procedure can distinguish between the hypotheses stating that the observed first order information on $f$ comes from $F_0$
{or} from $F_1$, with risk (the maximal probability of rejecting the true hypothesis) $\alpha$. On the other hand,
when $F=F_0$ or $F=F_1$, our observations  are deterministic functions of the samples $\omega_1$,...,$\omega_N$ drawn from the 2-dimensional normal distribution
  $\cN\left(\left[\begin{array}{c}0\\0\end{array}\right],\;\left[\begin{array}{cc}\sigma^2&0\\0&s^2\end{array}\right]\right)$
 for $F=F_0$, and $\cN\left(\left[\begin{array}{c}\delta\\d\end{array}\right],\;\left[\begin{array}{cc}\sigma^2&0\\0&s^2\end{array}\right]\right)$ for  $F=F_1$.
It is well known that deciding between such hypotheses with risk $\leq \alpha$ is possible only if
$$
\sqrt{{\delta^2\over \sigma^2}+{d^2\over s^2}}\geq {2\over\sqrt{N}} q_\N(1-\alpha).$$
%$\hbox{ErfInv}(\cdot)$ being the inverse error function: ${\ErfInv}(\alpha))=t$, $0<\alpha<1$, where
%\[
%\alpha=(2\pi)^{-1/2}\int_t^\infty \exp\{-s^2/2\}ds.
%\]
We arrive at the following lower bound on $W$, given $M_1$, $M_2$, with $M_1\geq M_2>0$:
\bse
W&\geq& \max_{{\delta\geq0,\;d\geq0}}\left\{{\delta+d}:\; \sqrt{{\delta^2\over \gamma^2M_2^2}+{d^2\over\gamma^2(M_1^2-M_2^2)}}\leq
  {2\over\sqrt{N}} q_\N(1-\alpha)\right\}
  %\\&=&
  ={2\gamma M_1 \over\sqrt{N}} q_\N(1-\alpha) = {\underbar{W}}.~~~~\square
\ese
\subsection{Proof of Lemma \protect{\ref{vincentlemma}}}
Without loss of generality we may assume that $\Opt=0$. Let $\bar{x}$ be such that $f_i(\bar{x})\leq-\varkappa$, $1\leq i\leq m$. Given $\delta>0$, there
 exists $x_\delta\in X$ such that $f_0(x_\delta)+\delta \leq \Phi(-\delta)$ and $f_i(x_\delta)\leq \Phi(-\delta)$, $1\leq i\leq m$; note that $\Phi(-\delta)>0$ due to $-\delta<0=\Opt$. The point
 \[
 x={\Phi(-\delta)\over \varkappa+\Phi(-\delta)}\bar{x}+{\varkappa\over \varkappa+\Phi(-\delta)}x_\delta
  \]
  belongs to $X$ and is feasible for (\ref{CSP}), since for $i\geq1$ one has
\[
f_i(x)\leq {\Phi(-\delta)\over \varkappa+\Phi(-\delta)}f_i(\bar{x})+{\varkappa\over \varkappa+\Phi(-\delta)}f_i(x_\delta)\leq -{\Phi(-\delta)\varkappa\over \varkappa+\Phi(-\delta)}+{\varkappa\Phi(-\delta)\over \varkappa+\Phi(-\delta)}=0.
\]
As a result,
\[
0=\Opt\leq f_0(x)\leq {\Phi(-\delta)\over \varkappa+\Phi(-\delta)}f_0(\bar{x})+{\varkappa\over \varkappa+\Phi(-\delta)}f_0(x_\delta)\leq {\Phi(-\delta)V\over \varkappa+\Phi(-\delta)}+{\varkappa\over \varkappa+\Phi(-\delta)}[\Phi(-\delta)-\delta].
\] The resulting inequality implies
$(\Phi(-\delta)-\Phi(0))/\delta=\Phi(-\delta)/\delta\geq\varkappa/(\varkappa+V)$; when passing to the limit as $\delta\to+0$, we get $\vartheta \geq\varkappa/(V+\varkappa)$.\qed
\subsection{Proof of Proposition \protect{\ref{propmain1}}}
Let us fix parameters $N$, $s$, $\lambda$, $\mu$ satisfying the premise of the proposition, let $\epsilon$, $\delta$ be associated with these parameters according to (\ref{epsilon}). We denote
\[\begin{array}{l}
\bar{f}_{0,N}(x,\xi^N)=f_{0,N}(x,\xi^N)-\mu M_1 N^{-1/2},\\
\bar{f}_{i,N}(x,\xi^N)=f_{i,N}(x,\xi^N)-\mu M_1 N^{-1/2},\;\;1\leq i\leq m,
\end{array}
\]
and set
$$
\overline{\Phi}_N(r,\xi^N)=\min_{x\in X}\max\left[\bar{f}_{0,N}(x,\xi^N)-r,\bar{f}_{1,N}(x,\xi^N),...,\bar{f}_{m,N}(x,\xi^N)\right].
$$
Then $\overline{\Phi}_N(r,\xi^N)$ is a convex nonincreasing function of $r\in \bR$ such that
$$
\Opt_N(\xi^N)=\min\{r: \overline{\Phi}_N(r,\xi^N)\leq0\}.
$$
Finally, let $\bar{r}$ be the smallest $r$ such that $\Phi(r)\leq\epsilon$. Since (\ref{CSP}) is feasible and $\Phi(r)\to\infty$ as $r\to-\infty$, $\bar{r}$ is a well defined real which is $<\Opt$ (since $\Opt$ is the smallest root of $\Phi$) and satisfies $\Phi(\bar{r})=\epsilon$.
\par
Let us set
$$
\widehat{\Xi}=\underbrace{\big\{\xi^N: \overline{\Phi}_N(\Opt,\xi^N)\leq0\big\}}_{\Xi_1}\cap\underbrace{\big\{\xi^N: \overline{\Phi}_N(\bar{r},\xi^N)>0\big\}}_{\Xi_2}.
$$
Since $\overline{\Phi}_N(r,\xi^N)$ is a nonincreasing function of $r$ and $\Opt_N(\xi^N)$ is the smallest root of $\overline{\Phi}_N(\cdot,\xi^N)$, for $\xi^N\in \widehat{\Xi}$ we have $\bar{r}\leq \Opt_N(\xi^N)\leq\Opt$. The left inequality here implies that $\Phi(\Opt_N(\xi^N)) \leq\epsilon$ (recall that $\Phi$ is nonincreasing and $\Phi(\bar{r})=\epsilon$). The bottom line is that when $\xi^N\in \widehat{\Xi}$, $\Opt_N(\xi^N)$ 
$\epsilon$-underestimates $\Opt$. Consequently, all we need to prove is that $\xi^N\not\in \widehat{\Xi}$ with probability at most $\delta$.\\
\par{1$^0$.} Let $x_*$ be an optimal solution to (\ref{CSP}). Same as in the proof of Proposition \ref{propmain}, for every $i$, $0\leq i \leq m$, we have (see (\ref{weget1}))
$$
\Prob\big\{f_{i,N}(x_*,\xi^N)>f_i(x_*)+\mu M_1 N^{-1/2}\big\}\leq\e^{-{\mu^2\over 4\tau_*}},
$$
whence for the event
$$
{\Xi'}=\big\{\xi^N:f_{i,N}(x_*,\xi^N)\leq f_i(x_*)+\mu M_1 N^{-1/2},\;0\leq i\leq m\big\}
$$
it holds
\be
\Prob\big\{\xi^N\not\in {\Xi'}\big\}\leq(m+1)\e^{-{\mu^2\over 4\tau_*}}.
\ee{star1}
By the origin of $x_*$ we have $f_0(x_*)\leq\Opt$ and $f_i(x_*)\leq0$, $1\leq i\leq m$. Therefore, for $\xi^N\in {\Xi'}$ it holds $\bar{f}_{0,N}(x_*,\xi^N)\leq \Opt$ and
$\bar{f}_{i,N}(x_*,\xi^N)\leq0$, $1\leq i\leq m$, that is,
$$\overline{\Phi}_N(\Opt,\xi^N)\leq\max[\bar{f}_{0,N}(x_*,\xi^N)-\Opt,\bar{f}_{1,N}(x_*,\xi^N),...,\bar{f}_{m,N}(x_*,\xi^N)]\leq0,$$
implying that $\xi^N\in \Xi_1$. We conclude that ${\Xi'} \subset \Xi_1$, and, by \rf{star1},
\begin{equation}\label{eqXi1}
\Prob\{\xi^N\not\in\Xi_1\}\leq (m+1)\e^{-{\mu^2\over 4\tau_*}}.\\
\end{equation}
\par{2$^0$.}  We have $\epsilon=\Phi(\bar{r})=\min_{x\in X} \max[f_0(x)-\bar{r},f_1(x),...,f_m(x)]$, whence by von Neumann's Lemma there exist nonnegative $y_i\geq0$, $0\leq i\leq m$,
summing up to 1, such that
$$
\begin{array}{rcl}
\epsilon&=&\min_{x\in X}[\ell(x):=y_0(f_0(x)-\bar{r})+\sum_{i=1}^m y_if_i(x)]\\
&=&\min_{x\in X}\Big[\int_\Xi\underbrace{\big[y_0[F_0(x,\xi)-\bar{r}]+{\sum}_{i=1}^m y_i F_i(x,\xi)\big]}_{\cL(x,\xi)}P(d\xi)\Big].\\
\end{array}
$$
Under the premise of the proposition, the integrand $F$ satisfies all assumptions of Proposition \ref{propmain}. Setting
$$
\ell_N(x,\xi^N)={1\over N}\sum_{i=1}^N\cL(x,\xi_i)
$$
and applying Proposition \ref{propmain} we get
\[\begin{array}{l}
\Prob\left\{\xi^N: \min_{x\in X} \ell_N(x,\xi^N)<\underbrace{{\min}_{x\in X}\ell(x)}_{=\epsilon}-\left[\mu  M_1+\left[\Omega [1+s^2]+2\lambda\right]{M_2}R\right]N^{-1/2}\right\}\\
~~\leq
\e^{-N(s^2-1)}+\e^{-{\mu^2\over 4\tau_*}}+\e^{-{\lambda^2\over 4\tau_*}}.
\end{array}\]
Now, in view of
$$\ell_N(x,\xi^N)=\underbrace{\lambda_0[\bar{f}_{0,N}(x,\xi^N)-\bar{r}]+\sum_{i=1}^m\lambda_i  {\bar f_{i, N}} (x,\xi^N)}_{\bar{\ell}_N(x,\xi^N)}+\mu M_1 N^{-1/2},$$
and due to the evident relation
$\min_{x\in X}\bar{\ell}_N(x,\xi^N) \leq \overline{\Phi}_N(\bar{r},\xi^N)$, we get
\bse
\lefteqn{\Prob\left\{\overline{\Phi}_N(\bar{r},\xi^N)<{\epsilon- \left[\mu  M_1+\left[\Omega [1+s^2]+2\lambda\right]{M_2}R\right]N^{-1/2} -\mu M_1 N^{-1/2}}\right\}}\\
&\leq&\Prob\left\{\min_{x\in X}\ell_N(x,\xi^N)<\epsilon-2N^{-1/2}\left[\mu M_1+M_2 {R}\left[{\Omega\over 2}[1+s^2]+\lambda\right]\right]\right\}\\
&\leq& \e^{-{\mu^2\over 4\tau_*}}+\e^{-N(s^2-1)}+\e^{-{\lambda^2\over 4\tau_*}}.
\ese
By (\ref{epsilon}), we have
\[
\epsilon-2\left[\mu M_1+M_2{R}\left[{\Omega\over 2}[1+s^2]+\lambda\right]\right]N^{-1/2}>0,
\] and we arrive at
$$
\Prob\big\{\xi^N\not\in\Xi_2\big\}=\Prob\big\{\overline{\Phi}_N(\bar{r},\xi^N)\leq0\big\}\leq \e^{-{\mu^2\over 4\tau_*}}+\e^{-N(s^2-1)}+\e^{-{\lambda^2\over 4\tau_*}}.
$$																																																															The latter bound combines with (\ref{eqXi1}) to imply the desired relation
$$
\Prob\big\{\xi^N\not\in\Xi\big\}\leq \e^{-{\mu^2\over 4\tau_*}}+\e^{-N(s^2-1)}+\e^{-{\lambda^2\over 4\tau_*}}+(m+1)\e^{-{\mu^2\over 4\tau_*}}=\beta.\eqno{\hbox{\qed}}
$$
%\section{Lower bounding SAA performance}\label{sec:lowersaa}

\section{Evaluating approximation parameters}\label{B-appendix}
For the sake of completeness we provide here the straightforward derivations of the parameter estimates used to build the bounds in the numerical section.
\subsection{Notation}
Let $P$ be a Borel probability distribution on $\bR^k$  and let $\Xi$ be the support of $P$. Consider the space $\cC$
of all Borel functions $g(\cdot):\Xi\to \bR$ such that $\bE_{\xi\sim P} \{\exp\{g^2(\xi)/M^2\} \}  <\infty$ for some $M=M(g)$. For $g\in \cC$, we set
\begin{equation} \label{orliczseminorm}
\pi[g]=\inf\big\{M \geq 0:\; \bE_{\xi\sim P} \{\exp\{g^2(\xi)/M^2\} \} \leq\exp\{1\}\big\}.
\end{equation}
It is well known \cite{van2013bernstein} that $\cC$ is a linear subspace in the space of real-valued Borel functions on $\Xi$ and $\pi[\cdot]$ is a semi-norm on this (Orlicz) space. Besides,
for a constant $g(\cdot)\equiv a$ we have $\pi[g]=|a|$, and $|g(\cdot)|\leq |h(\cdot)|$ with $h\in\cC$ and Borel $g$ implies $g\in \cC$ and $\pi[g]\leq\pi[h]$.
\par
Given a convex compact set $X\subset\bR^n$, a norm $\|\cdot\|$ on $\bR^n$, and a continuously differentiable distance-generating function $\omega(\cdot)$
for the unit ball $B_{\|\cdot\|}$ which  is compatible with this norm, let $R$ be the radius of the smallest  $\|\cdot\|$-ball containing $X$. Given a Borel function $F(x,\xi):\bR^n\times \Xi\to\bR$
which is convex in $x\in\bR^n$ and $P$-summable in $\xi$ for every $x$, let
\[
f(x)=\bE \{F(x,\xi)\}:\;X\to\bR.
\]
We set
$$
\begin{array}{rcl}
M_{1,\infty}&=&\sup_{x\in X, \xi\in\Xi}|F(x,\xi)-f(x)|,\\
M_{1,\exp}&=&\sup_{x\in X}\pi[F(x,\cdot)-f(x)],\\
L(x,\xi)&=&\sup_{g\in \partial_xF(x,\xi),h\in\partial f(x)}\|g-h\|_*,\\
M_2&=&\sup_{x\in X}\pi[L(x,\cdot)].
\end{array}
$$
%\begin{remark}\label{rem1}
Note that adding to $F(x,\xi)$ a differentiable function $g$ of $x$: $F(x,\xi)\mapsto F(x,\xi)+g(x)$  does not affect the quantities $M_{1,\infty}$, $M_{1,\exp}$, and $M_2$.
%\end{remark}

Our goal is to compute upper bounds on $M_{1,{\infty}}$, $M_{1,\exp}$, and $M_2$ in the different settings
 of Section \ref{stoquadopt}.

\subsection{Quadratic risk minimization}
In this case
\begin{itemize}
\item $X=\{x=[x_1;...;x_n]\in\bR^{n}: x_1,...,x_{n}\geq0,\sum_{i=1}^{n} x_i=1\}$,
\item $\Xi$ is a part of the unit box $\{\xi=[\xi_1;...;\xi_n]\in\bR^n:\|\xi\|_\infty\leq 1\}$,
\item $F(x,\xi)=\kappa_0 {\xi^Tx} + {\kappa_1 \over2}\left({\xi^Tx}\right)^2$, with
$\kappa_1 \geq0$, and
$
f(x)=\kappa_0\mu^Tx + {\kappa_1 \over2}x^T\bE\{\xi\xi^T\}x,
$
where $\mu=\bE\{\xi\}$.
\end{itemize}
The parameters $M_1$, $M_2$, $R$ and $\Omega$ of construction can be set according to:
\be
M_{1}\leq 2|\kappa_0|+{\kappa_1 \over2},\;\;M_2= 2|\kappa_0|+\kappa_1 ,
\;\;R=1,\;\;
\Omega=\left\{\begin{array}{ll}1,&n= 1\\
\sqrt{2},& n=2\\
\ln(n)\sqrt{\frac{2\e}{1 + \ln(n)}},&n\geq3.\\
\end{array}\right..
\ee{firstb}
{\small \begin{quotation}
Indeed, for $\xi\in \Xi$ and $x\in X$, we get
$$
|F(x,\xi)-f(x)|\leq |\kappa_0|| (\xi-\mu)^Tx |+{\kappa_1 \over2}|x^T(V-\xi\xi^T)x|\leq
|\kappa_0|\|\xi-\mu\|_\infty+ \frac{\kappa_1}{2}
$$
(indeed, since $V$ is positive semidefinite with $\|V\|_\infty\leq1$ and $\|\xi\|_\infty\leq1$, we have $|x^T(V-\xi\xi^T)x|\leq1$ for all $x$ such that $\|x\|_1\leq 1$), and
$$
M_{1,\exp}\leq M_{1,\infty}\leq |\kappa_0|(1+\|\mu\|_\infty )+{\kappa_1 \over2}\leq 2|\kappa_0|+{\kappa_1 \over2}.
$$
Further, let us equip $\bR^n$ with the norm  $\|\cdot\|=\|\cdot\|_1$, so that $\|\cdot\|_*=\|\cdot\|_\infty$, and endow the unit ball of the norm with the distance generating function\footnote{For details, see e.g., 
\cite[Theorem 2.1]{nesterov2013first}}.
\begin{equation}\label{dgf.B}
\omega(x)={1\over p\gamma}\sum_{i=1}^n|x_i|^p,\;p=\left\{\begin{array}{ll}
2&\mbox{for }n\leq 2,\\
1+1/\ln(n)&\mbox{for }n\geq 3,\\
\end{array}\right.,\; \gamma= \left\{\begin{array}{ll}1,&n\leq1\\
{1\over 2},&n=2,\\
{1\over \e\ln(n)},&n\geq3\\
\end{array}\right.
\end{equation}
%see \cite{MLOPT},
resulting in
$\Omega=\sqrt{2\over p\gamma}
\;\;\mbox{ and }
\;\;R=1.
$
Now let $x\in X$ and $\xi\in\Xi$, and let $g$ be a subgradient  of $F(x,\xi)$
with respect to $x$, and $h$ be a subgradient of $f$ at $x$. We have
\[g=\kappa_0\xi+\kappa_1 \xi(\xi^Tx),\;\;h=\kappa_0\mu+\kappa_1 Vx,\]
thus $$
\|g-h\|_*\leq| \kappa_0|\|\xi-\mu\|_\infty + \kappa_1   \|V - \xi \xi^\transp \|_{\infty}  \leq  |\kappa_0|(1+\|\mu\|_\infty)+ 2\kappa_1.
$$
We conclude that
\[
M_2\leq |\kappa_0|(1+\|\mu\|_\infty)+2\kappa_1 \leq 2|\kappa_0|+2\kappa_1.
\]
%and $L\leq |\kappa_0| + \kappa_1 $.
\end{quotation}
} %endof small
\subsection{Gaussian VaR optimization}

Here the situation is as follows:
\begin{itemize}
\item $X=\{x=[x_1;...;x_n]\in\bR^{n}: x_1,...,x_{n}\geq0,\sum_{i=1}^{n} x_i=1\}$,
\item $\xi\sim \cN(0,\Sigma )$ on $\bR^n$,\;$\Sigma\succ 0$,
\item $F(x,\xi)=\kappa_0{\xi^Tx} + \kappa_1 |{\xi^Tx}|$, with
$\kappa_1 \geq0$.
\end{itemize}
We have $
f(x)=\sqrt{{2\over\pi}}\kappa_1 \sigma_x,
$ with $\sigma_x=\sqrt{x^T\Sigma x}$.
In this case one can set $\Omega$ and $R$ as in \rf{firstb}, along with
\[\begin{array}{rcl}
M_{1}& = & \left[\sqrt{{2\e^2\over \e^2-1}}|\kappa_0|+\sqrt{2}\kappa_1 \right]\smax,\\
M_2&=&(|\kappa_0|+\kappa_1 )\smax \sqrt{2(2+\ln n)}+\kappa_1 \smax\sqrt{2\over \pi},
\end{array}
\]
where $\smax^2=\max_{1\leq i\leq n}\Sigma_{i, i}$.
{\small\begin{quotation}
Indeed, we have  $\xi^Tx \sim\cN(0,\sigma^2_x)$,
we conclude that
$
f(x)=\kappa_1 \sqrt{{2\over\pi}}\sigma_x,
$
whence
$$
|F(x,\xi)-f(x)|\leq |\kappa_0||\xi^Tx|+\kappa_1 ||\xi^T x|-\sqrt{2/\pi}\sigma_x|=
\sigma_x\left[|\kappa_0||\eta_x|+\kappa_1 ||\eta_x|-\sqrt{2/\pi}|\right]
$$
where $\eta_x=\xi^Tx /\sigma_x\sim \cN(0,1)$.
By direct computation we get
$$
\pi[|\eta_x|]=\nu:=\sqrt{{2\e^2\over \e^2-1}}= 1.52...
$$
Next, setting $\vartheta=\sqrt{2/\pi}$ we observe that
$$
\begin{array}{l}
{1\over\sqrt{2\pi}}\int\exp\{||s|-\vartheta |^2/2-s^2/2\}ds=\vartheta \int_0^\infty \exp\{[s^2-2 \vartheta s+ \vartheta^2-s^2]/2\}ds\\
=
\vartheta \int_0^\infty\exp\{ \vartheta^2/2- \vartheta s\}ds=\exp\{ \vartheta^2/2\}<\exp\{1\},
\end{array}
$$
implying that
$$
\pi[||\eta_x|-\sqrt{2/\pi}|] \leq \sqrt{2}.
$$
As a result,
$$
\pi[F(x,\cdot)-f(x)] \leq \sigma_x\left[|\kappa_0|\pi[| \eta_x |  ] +\kappa_1 \pi[||\eta_x|-\sqrt{2/\pi}|]\right]\leq \sigma_x\left[\nu |\kappa_0|+ \sqrt{2}\kappa_1 \right].
$$
Taking into account that for all $x\in X$ $\sigma_x^2=x^T\Sigma x\leq \|\Sigma \|_\infty$, we arrive at
\begin{equation}\label{M1.C}
M_{1,\exp}\leq  \left[\nu|\kappa_0|+ \sqrt{2}\kappa_1 \right]\sqrt{\|\Sigma \|_\infty}=\left[\nu|\kappa_0|+\sqrt{2}\kappa_1 \right]\smax.
\end{equation}

Let $x\in X$, and let $g$ be a subgradient
with respect to $x$ of $F(x,\xi)$, and $h$ be a subgradient of $f(x)$. We  have
$$
g=\kappa_0\xi+\kappa_1 \xi\chi
$$
with $\chi=\chi(x,\xi)\in[-1,1]$, so that
$$
\|g\|_\infty\leq [|\kappa_0|+\kappa_1 ]\|\xi\|_\infty.
$$
Note that
\[
\partial \left[\sqrt{x^T\Sigma x}\right]=\left\{\begin{array}{ll}\left\{\left(x^T\Sigma x\right)^{-1/2}\Sigma x\right\},&x\neq0,\\
\left\{\Sigma^{1/2}u,\;\|u\|_2\leq 1
\right\},&x=0.\end{array}
\right.
\]
Therefore,  for all $h\in \partial f(x)$ one has
%$$
%\begin{array}{lll}
%\|h\|_\infty &\leq& \kappa_1 \sqrt{2\over \pi}\sup_{x\neq 0} {\|\Sigma x\|_\infty\over \sqrt{x^T\Sigma x}}\\
%&=&
%\kappa_1 \sqrt{2\over \pi}\sup_{y\neq 0}{\|\Sigma^{1/2} y\|_\infty\over \|y\|_2}
%=
%\kappa_1 \sqrt{2\over \pi} \displaystyle  \max_{1\leq i\leq n} {\color{red}\|\Sigma_i^{1/2}\|_{2} }\\
%&= &{\color{red}\kappa_1 \sqrt{2\over \pi} \displaystyle \max_{1\leq i\leq n} \sqrt{\Sigma_{i,i}} = }\kappa_1 \smax\sqrt{2\over \pi},
%\end{array}
%$$
\[
\|h\|_\infty\leq \kappa_1 \sqrt{2\over \pi}\sup_{x\neq 0} {\|\Sigma x\|_\infty\over \sqrt{x^T\Sigma x}}=
\kappa_1 \sqrt{2\over \pi}\sup_{y\neq 0}{\|\Sigma^{1/2} y\|_\infty\over \|y\|_2}=
\kappa_1 \sqrt{2\over \pi}  \max_{1\leq i\leq n}\|\Sigma^{1/2}_i\|_2 \leq \kappa_1 \smax\sqrt{2\over \pi}
\]
(here $\Sigma^{1/2}_i$ stands for the $i$-th row of $\Sigma^{1/2}$),
and
%\[
%\partial f(x)=\left\{\begin{array}{rl}\left\{\kappa_1 \sqrt{2\over \pi}{\Sigma x\over \sqrt{x^T\Sigma x}}\right\},&x\neq0,\\
%\left\{\Sigma^{1/2}u,\;\|u\|_2\leq 1
%\right\},&x=0.\end{array}
%\right.
%\]
%$f$ is differentiable at $x\neq 0$ with $\nabla f(x)=\kappa_1 \sqrt{2\over \pi}{\Sigma x\over \sqrt{x^T\Sigma x}}$ for $x\neq 0$. Therefore, the Lipschitz constant of $f$ with respect to $\|\cdot\|_1$ does not exceed
%\[
%\kappa_1 \sqrt{2\over \pi}\sup_{x\neq 0} {\|\Sigma x\|_\infty\over \sqrt{x^T\Sigma x}}=
%\kappa_1 \sqrt{2\over \pi}\sup_{y\neq 0}{\|\Sigma^{1/2} y\|_\infty\over \|y\|_2}=
%\kappa_1 \sqrt{2\over \pi}\|\Sigma^{1/2}\|_{2,\infty}\leq \kappa_1 \sqrt{{2\over \pi}\max_{i}\Sigma_{ii}}
%\]
%Further, because
%$
%f(x)=\kappa_1  \int |\xi^Tx|P(d\xi),
%$
%the Lipschitz constant of $f$ with respect to $\|\cdot\|_1$ does not exceed
%$\kappa_1  \bE\{\|\xi\|_\infty\}.
%$
%Thus, $h$ satisfies $\|h\|_\infty\leq \kappa_1  \bE\{\|\xi\|_\infty\}$, and therefore
$$
\|g-h\|_*=\|g-h\|_\infty\leq [|\kappa_0|+\kappa_1 ]\|\xi\|_\infty+\kappa_1  \smax\sqrt{2\over \pi}, % \bE\{\|\xi\|_\infty\},
$$
that is,
$$
L(x,\xi)\leq [|\kappa_0|+\kappa_1 ]\|\xi\|_\infty+\kappa_1  \smax\sqrt{2\over \pi}.
$$
We conclude that
\be
\pi[L(x,\cdot)]\leq [|\kappa_0|+\kappa_1 ]\pi[\|\xi\|_\infty]+\kappa_1 \smax  \sqrt{2\over \pi}.
%\leq
%[|\kappa_0|+\kappa_1 ]\smax \sqrt{2(2+\ln n)}+\kappa_1 \smax \sqrt{2\over \pi},
\ee{pixi}
We now use the following simple result.\footnote{In fact, in the numerical experiments we have used a slightly better bound $M_2$ which can be defined as follows. Let $t_n$, $0<t_n<\smax$ be the unique solution of the equation
\[
{\tilde h}_n(t_n) = {n^{2t_n^2 \smax^2}\over 1-2t_n^2 \smax^2}=\e
\]
(observe that $\tilde h_n(\cdot)$ is monotone on $]0,\frac{1}{\sqrt{2} \smax}[$, so $t_n$ can be computed using bisection). The same reasoning as in the proof of Lemma \ref{lemgaussespinf} results in the bound
\begin{equation}\label{valueM2marko}
M_2 = \frac{(|\kappa_0|+\kappa_1 )}{ t_n } +  \kappa_1 \smax\sqrt{2\over \pi}.
\end{equation}
For instance, in the experiments of Section \ref{portfolioproblem}, for $\smax=\sqrt{6}$ and
$n \in \{2,10, 20, 100\}$, the values of $1/t_n$ (resp. of its upper bound $\smax \sqrt{2(2+\ln n)}$)
were $4.97, 6.46, 7.05, 8.27$ (resp. $5.68, 7.19, 7.74, 8.90$).
}

\begin{lemma}\label{lemgaussespinf} Let $\xi$ be a zero-mean Gaussian random vector in $\bR^n$, and let $\bar{\sigma}^2 \geq \max_{1\leq i\leq n}\bE\{\xi^2_i\}$. Then for $M\geq \bar{\sigma} \sqrt{2(2+\ln n)}$
$$
\bE\big\{ \e^{\|\xi \|_{\infty}^2/M^2 } \big\} \leq
\e.$$
\end{lemma}
\proof
Let $\eta_n=\max_{1\leq i\leq n} |\xi_i|$. We have the following well-known fact:
\[
\psi_n(r):=\Prob \{\eta_n\geq r\}\leq \min\left\{1,n\e^{-{r^2\over 2\bar{\sigma}^2}}\right\}.
\]
Therefore, for $|t|<(\sqrt{2}\bar{\sigma})^{-1}$,
\bse
 \bE\big\{\e^{t^2\eta^2_n}\big\}&=&-\int_0^\infty \e^{t^2r^2} d\psi_n(r)=
 1+\int_{0}^{\infty} 2t^2r\e^{t^2r^2}\psi_n(r)dr\\
 &\leq& \e^{2t^2\bar{\sigma}^2\ln n}
 +{2nt^2}\int_{\bar{\sigma}\sqrt{2\ln n}}^\infty r\exp\Big\{-{(1-2t^2\bar{\sigma}^2)r^2\over 2\bar{\sigma}^2}\Big\}dr\\
 &=& \e^{2t^2\bar{\sigma}^2\ln n}+{2t^2\bar{\sigma}^2\over 1-2t^2\bar{\sigma}^2}\e^{2t^2\bar{\sigma}^2\ln n}={n^{2t^2\bar{\sigma}^2}\over 1-2t^2\bar{\sigma}^2}.
 \ese
 Note that $\e^{-x}\leq 1-x/2$ for $0\leq x\leq 1$.  Thus for all $n\geq 1$ and $t\leq \big(\bar{\sigma}\sqrt{2(2+\ln n)}\big)^{-1}$,
 \[
 {n^{2t^2\bar{\sigma}^2}\over 1-2t^2\bar{\sigma}^2}\leq  {\e^{\ln n\over 2+\ln n}
 \over 1-{1\over {2+\ln n}}}=\e^{1-{2\over {2+\ln n}}}{2+\ln n\over 1+\ln n}
 \leq \e.\eqno{\mbox{\qed}}
 \]

Finally, using the result of the lemma we conclude from \rf{pixi} that one can take for $M_2$ the expression
\begin{equation}\label{firstM2marko}
(|\kappa_0|+\kappa_1 )\smax \sqrt{2(2+\ln n)}+\kappa_1 \smax\sqrt{2\over \pi}.
\end{equation}
\end{quotation}}

\subsection{CVaR optimization}

Consider the portfolio problem of Section \ref{portboundedreturns}. With some terminology abuse, in what follows,  we refer
to  the special case  $n=1$ with $x_1\equiv 1$ as to the case of $n=0$.
\begin{itemize}
\item $X=\{x=[x_0;x_1;...;x_n]\in\bR^{n+1}: \;|x_0|\leq 1,\, x_1,...,x_{n}\geq0,\,\sum_{i=1}^{n} x_i=1\}$,
\item $\Xi$ be a part of the unit box $\{\xi=[\xi_1;...;\xi_n]\in\bR^n:\|\xi\|_\infty\leq 1$\},
\item $F(x,\xi)=\kappa_0\sum_{i=1}^n\xi_ix_i + \kappa_1 \left[x_0+{1\over\epsilon}[\sum_{i=1}^n\xi_ix_i-x_0]_+\right]$, with
$\kappa_0,\kappa_1 \in[0,1]$.
\end{itemize}
The parameters $M_1$, $M_2$, $R$ and $\Omega$ of construction can be set according to:
\bse
& M_{1} =   2\left(\kappa_0+{\kappa_1 \over\epsilon}\right),\;\;&M_2= \left\{\begin{array}{ll}{\kappa_1 \over\epsilon},&n=0,\\
\sqrt{\left({\kappa_1 \over\epsilon}\right)^2+4 \left(\kappa_0+ {\kappa_1 \over\epsilon}\right)^2},&n\geq1,
\end{array}\right.\\
&
R=\left\{\begin{array}{ll}1,&{n=0}\\
\sqrt{2},&{n\geq 1}.
\end{array}\right.,\;\;
&\Omega=\left\{\begin{array}{ll}1,&n=0,\\
\sqrt{2},& n=1,\\
\sqrt{3},& n=2,\\
\sqrt{1 + \frac{2 e  (\ln (n))^2 }{1 + \ln (n)}},&n\geq3.\\
\end{array}\right.
\ese
{\small\begin{quotation}
Indeed, denoting ${\xi_x}=\xi^T x$ and $\mu_i=\bE\{\xi_i\}$, we have
$$
f(x)=\kappa_0\sum_{i=1}^{{n  } }\mu_ix_i +\kappa_1 \left[x_0+{1\over\epsilon}\bE\{[{\xi_x}-x_0]_+\}\right],
$$
whence for $\xi\in \Xi$ and $x\in X$
$$
|F(x,\xi)-f(x)|\leq \kappa_0|\sum_{i=1}^n [\xi_i-\mu_i]x_i|+{\kappa_1 \over\epsilon}|[{\xi_x}-x_0]_+-\bE\{[{\xi_x}-x_0]_+\}|.
$$
We have  $|{\xi_x}|\leq1$, whence
{
$0 \leq [{\xi_x}-x_0]_+\leq 1+[-x_0]_+$ and
$
0  \leq \bE\{[{\xi_x}-x_0]_+\}\leq 1+[-x_0]_+.
$
Then,
$$
-2 \leq -1-[-x_0]_+ \leq [{\xi_x}-x_0]_+-\bE\{[{\xi_x}-x_0]_+\}\leq 1+[-x_0]_+,
$$
so that
$$
|[{\xi_x}-x_0]_+-\bE\{[{\xi_x}-x_0]_+\}|\leq 2.
$$
}
\if{
$-1+[-x_0]_+\leq [{\xi_x}-x_0]_+\leq 1+[-x_0]_+$, and
$$
-1+[-x_0]_+\leq \bE\{[{\xi_x}-x_0]_+\}\leq 1+[-x_0]_+.
$$
Then,
$$
-1+[-x_0]_+-[1+[-x_0]_+]\leq [{\xi_x}-x_0]_+-\bE\{[{\xi_x}-x_0]_+\}\leq 1+[-x_0]_+-[-1+[-x_0]_+],
$$
so that
$$
|[{\xi_x}-x_0]_+-\bE\{[{\xi_x}-x_0]_+\}|\leq 2.
$$
}\fi

We conclude that
\begin{equation}\label{M1.A}
M_{1,\exp}\leq M_{1,\infty}\leq \kappa_0(1+\|\mu\|_\infty)+{2\kappa_1 \over\epsilon}\leq 2\left(\kappa_0+{\kappa_1 \over\epsilon}\right).
\end{equation}
In what follows, for a vector from $\bR^{n+1}$, say, $z=[z_0;z_1;...;z_n]$, we set $z'=[z_1;...;z_n]$, so that $z=[z_0;z']$.  Let us define norm $\|\cdot\|$ on $\bR^{n+1}$ as
$$
\|[x_0;x']\|=\sqrt{x_0^2+\|x'\|_1^2},
$$
implying that
$$
\|[x_0;x']\|_*=\sqrt{x_0^2+\|x'\|_\infty^2}.
$$
A distance-generating function $\omega([x_0;x'])$ for the unit ball of the norm $\|\cdot\|$ can be taken as
$$
\omega([x_0;x'])={1\over 2}x_0^2+{1\over p\gamma}\sum_{i=1}^n|x_i|^p,\;p=\left\{\begin{array}{ll}
2,&n\leq 2\\
1+1/\ln(n),&n\geq 3\\
\end{array}\right., \;\gamma= \left\{\begin{array}{ll}1,&{n\leq1}\\
{1\over 2},&n=2,\\
{1\over \e\ln(n)},&n\geq3,\\
\end{array}\right.
$$
resulting in
%{$\Omega=\sqrt{1+{2\over p\gamma}}$ and $R=\sqrt{2}$}.

\begin{equation}\label{Omega.A}
{
\Omega=\left\{\begin{array}{ll}1,& {n=0}\\
\sqrt{{1}+{2\over p\gamma}},& {n\geq1} ,
\end{array}\right.
}
%\end{equation}
 \;\;\mbox{and}\;\;
%\begin{equation}\label{D.A}
{
R=\left\{\begin{array}{ll}1,&{n=0}\\
\sqrt{2},&{n\geq 1}.
\end{array}\right.
}
\end{equation}
Let $x\in X$ and $\xi\in\Xi$, and let $g=[g_0;g']$ be a subgradient  of $F(x,\xi)$ with respect to $x$, and $h$ be a subgradient of $f$ at $x$. We clearly have
%$$
%\begin{array}{rcl}
%g_0&=&\kappa_1 -\kappa_1 \over\epsilon}\chi_0,\\
%g'&=&\kappa_0\xi+\kappa_1 \over\epsilon}\xi \chi_1,\\
%h_0&=&\kappa_1 -\kappa_1 \over\epsilon}\chi_2\\
%\end{array}
%$$
\[
g_0=\kappa_1 -{\kappa_1 \over\epsilon}\chi_0,\;\;
g'=\kappa_0\xi+{\kappa_1 \over\epsilon}\xi \chi_1,\;\;
h_0=\kappa_1 -{\kappa_1 \over\epsilon}\chi_2,
\]
where $\chi_i\in[0,1]$.
Next, for $n \geq 2$,
$$
\begin{array}{lll}
|f([x_0;x']) - f([x_0;y'])|&=&
| \kappa_0 \mu^T (x' - y') + \frac{\kappa_1 }{\varepsilon} \Big( \bE\{[\xi^T x'   -x_0]_+\} - \bE\{[\xi^T y'   -x_0]_+\} \Big)  |\\
& \leq & \kappa_0 \|\mu\|_{\infty} \|x' - y'\|_1 + \frac{\kappa_1 }{\varepsilon}\bE\{ |\xi^T(x' -y')|\}\\
& \leq & (\kappa_0 + \frac{\kappa_1 }{\varepsilon}) \|x' - y'\|_1.
\end{array}
$$
It follows that $f([x_0;x'])$ is Lipschitz continuous in $x'$ with constant $\kappa_0 + \frac{\kappa_1 }{\varepsilon}$
with respect to $\|\cdot\|_1$ and we have $\|h'\|_\infty\leq \kappa_0 + {\kappa_1 \over\epsilon}$.
As a result, we obtain for $n \geq 2$
$$
\|g-h\|_*=\sqrt{|g_0-h_0|^2+\|g'-h'\|_\infty^2}\leq \sqrt{\left({\kappa_1 \over\epsilon}\right)^2+4 \left(\kappa_0+{\kappa_1 \over\epsilon}\right)^2}
$$
while $\|g-h\|_*\leq {\kappa_1 \over\epsilon}$ for $n=1$.

We conclude that
\begin{equation}\label{M2.A}
M_2=\left\{\begin{array}{ll}{\kappa_1 \over\epsilon},&{n=0},\\
\sqrt{\left({\kappa_1 \over\epsilon}\right)^2+4\left(\kappa_0+{\kappa_1 \over\epsilon}\right)^2},
&{n\geq1}.
\end{array}\right.
\end{equation}
\if{
Finally, we have $|h_0|\leq |\kappa_1 /\epsilon-1|$ and, as we have seed, $\|h'\|_\infty\leq \kappa_1 /\epsilon$, whence
\begin{equation}\label{L.A}
L\leq {\sqrt{(\kappa_1 -\kappa_0\epsilon)^2+\kappa_1 ^2}\over\epsilon}.
\end{equation}
}\fi
\end{quotation}
}

\bibliographystyle{abbrv}
\bibliography{References}

\end{document}